\documentclass[a4paper,reqno,11pt,authoryear]{amsart}
\usepackage{fancyhdr}
\usepackage{dsfont}
\usepackage[mathscr]{eucal}
\usepackage[cp1251]{inputenc}
\usepackage[english]{babel}
\usepackage{cite,enumerate,float,indentfirst}
\usepackage{graphicx}
\usepackage{xcolor}
\usepackage{mathrsfs,amsthm,amsmath,url}
\usepackage{latexsym,a4,amssymb,amsfonts}
\usepackage{epstopdf}
\usepackage{mathtools}
\usepackage{etoolbox}
\usepackage{color}
\usepackage[textwidth=4cm,textsize=footnotesize]{todonotes}

\definecolor{blue(pigment)}{rgb}{0.2, 0.2, 0.6}
\definecolor{ultramarine}{rgb}{0.07, 0.04, 0.56}
\definecolor{darkspringgreen}{rgb}{0.09, 0.45, 0.27}
\definecolor{hookersgreen}{rgb}{0.0, 0.44, 0.0}
\definecolor{plum(traditional)}{rgb}{0.56, 0.27, 0.52}
\definecolor{purple(html/css)}{rgb}{0.5, 0.0, 0.5}
\definecolor{magenta(dye)}{rgb}{0.79, 0.08, 0.48}

\usepackage{hyperref}

\hypersetup{colorlinks,
	linkcolor=hookersgreen,
	linktoc=blue,
	citecolor=ultramarine,
	urlcolor=black,
	filecolor=black
}
\usepackage{nameref}

\numberwithin{equation}{section}

\newtheorem{theorem}{Theorem}[section]
\newtheorem{lemma}[theorem]{Lemma}

\newtheorem{definition}[theorem]{Definition}
\newtheorem{corollary}[theorem]{Corollary}
\theoremstyle{remark}
\newtheorem*{remark}{Remark}

\newcommand{\R}{\mathbb{R}}
\newcommand{\C}{\mathbb{C}}
\newcommand{\J}{\mathbb{J}}
\newcommand{\K}{\mathbb{K}}

\newcommand{\T}{\mathbb{T}}

\newcommand{\Cond}{{\bf (C0)}}
\newcommand{\CondTwo}{{\bf (C1)}}
\newcommand{\CondThree}{{\bf (CG)}}

\newcommand{\X}{{\bf X}}
\newcommand{\U}{{\bf U}}
\newcommand{\A}{{\bf A}}

\newcommand{\B}{{\bf B}}

\newcommand{\EE}{{\bf E}}

\newcommand{\G}{{\bf G}}

\newcommand{\HH}{{\bf H}}

\newcommand{\I}{{\bf I}}
\newcommand{\RR}{{\bf R}}
\newcommand{\RS}{{\bf S}}
\newcommand{\RT}{{\bf T}}
\newcommand{\V}{{\bf V}}

\newcommand{\y}{{\bf y}}

\newcommand{\LL}{{\bf L}}

\newcommand{\W}{{\bf W}}

\newcommand{\ee}{{\bf e}}
\newcommand{\uu}{{\bf u}}

\newcommand{\ve}{\varepsilon}

\newcommand{\tTj}{\widetilde T_{n}^{(j)}}

\DeclareMathOperator{\Tr}{Tr}
\DeclareMathOperator{\Rank}{Rank}

\DeclareMathOperator{\E}{\mathbb{E}}

\DeclareMathOperator{\Pb}{\mathbb{P}}

\DeclareMathOperator{\one}{\mathds{1}}
\DeclareMathOperator{\re}{Re}
\DeclareMathOperator{\imag}{Im}

\begin{document}

\vspace{1in}

\title[Local semicircle law]{\bf Local semicircle law under fourth moment condition}

\author[F. G{\"o}tze]{F. G{\"o}tze}
\address{Friedrich G{\"o}tze\\
 Faculty of Mathematics\\
 Bielefeld University \\
 Bielefeld, Germany
}
\email{goetze@math.uni-bielefeld.de}

\author[A. Naumov]{A. Naumov}
\address{Alexey A. Naumov\\
 National Research University Higher School of Economics, Moscow, Russia; and IITP RAS, Moscow, Russia
 }
\email{anaumov@hse.ru}

\author[A. Tikhomirov]{A. Tikhomirov}
\address{Alexander N. Tikhomirov\\
 Department of Mathematics\\
 Komi Science Center of Ural Division of RAS \\
 Syktyvkar, Russia; and National Research University Higher School of Economics, Moscow, Russia
 }
\email{tikhomirov@dm.komisc.ru}

%\thanks{}


\date{\today}

\begin{abstract}
We consider a random symmetric matrix \(\X = [X_{jk}]_{j,k=1}^n\) with upper triangular entries being independent random variables with mean zero and unit variance. Assuming that \( \max_{jk} \E |X_{jk}|^{4+\delta} < \infty, \delta > 0\), it was proved in \cite{GotzeNauTikh2016a} that with high probability the typical distance between the Stieltjes transforms \(m_n(z)\), \(z = u + i v\), of the empirical spectral distribution (ESD) and the Stieltjes transforms  \(m_{sc}(z)\) of the semicircle law is of order \((nv)^{-1} \log n\).  The aim of this paper is to remove \(\delta>0\) and show that this result still holds if we assume that \( \max_{jk} \E |X_{jk}|^{4} < \infty\). We also discuss applications to the rate of convergence of the ESD to the semicircle law in the Kolmogorov distance, rates of localization of the eigenvalues around the classical positions and rates of delocalization of eigenvectors.  

\end{abstract}

\maketitle

\section{Introduction and main result}
One of the main questions in random matrix theory is to investigate the limiting behaviour of spectral statistics of eigenvalues of large dimensional random matrices, for example, the distance between neighbouring eigenvalues or \(k\)-point correlation function. It turns out that there is a universality phenomena which states that the distribution of these statistics is independent of the particular distribution of the matrix entries, but depends on some global characteristics like the existence of moments. In the recent years there was a significant progress in the analysis of universality phenomena for the Wigner ensemble of random matrices, i.e. Hermitian matrices with independent entries subject
to the symmetry constraint. We refer the interested reader for a comprehensive literature review and more details to the forthcoming book by L. Erd\"os and H.-T. Yau~\cite{ErdosYauBOOK}. In the current paper we will not discuss the question of universality, but turn our attention to the {\it local semicircle law} which is the necessary intermediate step to the universality, but has its own important applications.

In what follows we consider a Hermitian random matrix \(\X: = [X_{jk}]_{j,k=1}^n\), such that \(X_{jk}, 1 \le j \le k \le n\) are independent random variables (r.v.) with zero mean. We also allow the distribution of matrix entries to depend on \(n\), but omit the latter from matrix notations. Furthermore, for simplicity we will assume that \( \E |X_{jk}|^2 = 1\) for all \(1 \le j \le k \le n\). As it was mentioned above we refer to such matrices as Wigner's ensemble. Denote the  eigenvalues of the normalized matrix \(\W: = n^{-1/2} \X\)  in the increasing order by \(\lambda_1 \le \ldots \le \lambda_n\)
and introduce the empirical spectral distribution (ESD) \(\mu_n:= n^{-1} \sum_{k =1}^n \delta_{\lambda_k}\). It was proved by E. Wigner~\cite{Wigner1955} and further generalized by many authors (see e.g. monographs~\cite{BaiSilv2010},~\cite{AndersonZeit},~\cite{Tao2012}) that with probability one
\(\mu_n\) weakly converges to the deterministic limit \(\mu_{sc}\) with absolutely continues density 
\begin{eqnarray}
g_{sc}(\lambda) := \frac{1}{2\pi} \sqrt{(4 - \lambda^2)_{+}},
\end{eqnarray}
where \((x)_{+}: = \max(x, 0)\). In particular, these results imply convergence in the {\it macroscopic regime}, i.e. for all intervals of fixed length and independent of \(n\), which contain macroscopically large number of eigenvalues. In turned out that an appropriate analytical tool is the Stieltjes transform of ESD \(\mu_n\) given by
\begin{eqnarray}\label{m_n def}
	m_n(z) := \int_{-\infty}^\infty \frac{\mu_n(d\lambda)}{\lambda - z},
\end{eqnarray}
where \(z = u + i v, v >0\). Under rather general conditions one may show (see e.g.~\cite{Tao2012}) that with probability one for {\it fixed} \(v > 0\)
\begin{eqnarray}\label{convergence}
\lim_{n \rightarrow \infty} m_n(z) = m_{sc}(z):= \int_{-\infty}^\infty \frac{g_{sc}(\lambda)\, d\lambda}{\lambda-z} = -\frac{z}{2} + \sqrt{\frac{z^2}{4} - 1}.
\end{eqnarray}
It is of  interest to investigate the {\it microscopic regime}, i.e. the case  of smaller intervals, where the number of eigenvalues cease to be macroscopically large. This regime is essential for many applications as the rate of convergence of \(\mu_n\) to the limiting distribution \(\mu_{sc}\), rigidity of eigenvalues \(\lambda_j, j = 1, \ldots, n,\) or delocalization of the corresponding eigenvectors \(\uu_j\) among others. To deal with this regime one needs to establish the convergence of \(m_n(z)\) to \(m_{sc}(z)\) in the region \(1 \geq v \geq f(n)/n\), where \(f(n) > 1\) is some function of \(n\). Significant progress in that direction was recently made in a series of results by L.~Erd{\"o}s, B.~Schlein, H.-T.~Yau, J.~Yin et al \cite{ErdosSchleinYau2009},~\cite{ErdosSchleinYau2009b},~\cite{ErdosSchleinYau2010},\cite{ErdosYauYin2012},~\cite{ErdKnowYauYin2013} showing that with high probability uniformly in \(u \in \R\)
\begin{eqnarray}\label{fluctuations of m_n around s}
|m_n(u+iv) - s(u+iv)| \le \frac{\log^{\alpha(n)} n }{nv},
\end{eqnarray}
where \(\alpha(n) := c \log \log n\) and \(c\) is some positive constant. This result was called  the {\it local semicircle law}. It means  that the fluctuations of \(m_n(z)\) around \(m_{sc}(z)\) are of order \((nv)^{-1}\) (up to a logarithmic factor). In the papers~\cite{ErdosSchleinYau2009},~\cite{ErdosSchleinYau2009b},~\cite{ErdosSchleinYau2010}, \cite{ErdosYauYin2012} the inequality ~\eqref{fluctuations of m_n around s}
has been proved  assuming that the distribution of \(X_{jk}\) has sub-exponential tails
for all \(1 \le j, k \le n\). Moreover in~\cite{ErdKnowYauYin2013} this assumption had been relaxed to requiring \(\beta_{p}: = \max_{j,k} \E |X_{jk}|^p \le C_p\) for all \(p \geq 1\), where \(C_p\) are some constants. In the recent years the series of results appeared, where the latter assumptions were further relaxed to the condition that 
\begin{eqnarray}\label{4+delta cond}
\beta_{4+\delta} < \infty 
\end{eqnarray} 
for some \(\delta>0\), see e.g.~\cite{ErdKnowYauYin2013a},~\cite{ErdKnowYauYin2012},~\cite{LeeYin2014},~\cite{GotzeNauTikh2015a},~\cite{GotzeNauTikh2015b},~\cite{GotzeNauTikh2018TVP},~\cite{GotTikh2015} and~\cite{GotzeNauTikh2016a}. In particular, the result of~\cite{GotzeNauTikh2016a} implies that~\eqref{fluctuations of m_n around s} holds with \(\alpha(n) \equiv 1\). 

The main emphasis of the current paper is to remove \(\delta\) from the condition~\eqref{4+delta cond}. The main idea of the proof is motivated by the recent result of A.~Aggarwal~\cite{Aggarwal2016} who established the bulk universality for Wigner's matrices with finite moments of order \(2 + \ve, \ve > 0\). He proved that~\eqref{fluctuations of m_n around s} still holds true, but the factor \((nv)^{-1}\) is replaced by \((nv)^{-1/2} + n^{-c \varepsilon}\), where \(c>0\) is some constant depending on \(\varepsilon\) and \(\beta_{2 + \varepsilon}\).  In the current paper we show that~\eqref{fluctuations of m_n around s} still holds assuming finite fourth moment only. Taking into account the behaviour of the extreme eigenvalues of \(\X\) we also believe that it is the best possible moment assumption for~\eqref{fluctuations of m_n around s} to remain valid. In the section~\ref{sketch} below we briefly discuss how using technique from \cite{Aggarwal2016} and~\cite{GotzeNauTikh2016a} one may achieve this aim.

\subsection{Notations} \label{sec: notation}
Throughout the paper we will use the following notations. We assume that all random variables are defined on common probability space $(\Omega, \mathcal F, \Pb)$ and let $\E$ be the mathematical expectation with respect to $\Pb$. For a r.v. \(\xi\) we use notation \(\E^{1/p} \xi\) to denote \( (\E \xi)^{1/p}\).   
We denote by $\one[A]$ the indicator function of the set $A$.

We denote by $\R$ and $\C$ the set of all real and complex numbers. Let \(\imag z, \re z\) be the imaginary and real parts of \(z\in \C\). We also define $\C^{+}: = \{z \in \C: \imag  z \geq 0\}$. Let $\T = [1, ... , n]$ denotes the set of the first $n$ positive integers. For any $\J \subset \T$ introduce $\T_{\J}: = \T \setminus \J$. To simplify all notations we will write $\T_j, \T_{\J, j}$ instead of $\T_{\{j\}}$ and $\T_{\J \cup \{j\}}$ respectively.

For any matrix $\W$ together with its resolvent $\RR$ and Stieltjes transform $m_n$
we shall systematically use the corresponding notations	$\W^{(\J)}, \RR^{(\J)}, m_n^{(\J)}$, respectively,  for the  sub-matrix  of $\W$ with entries $X_{jk}, j,k \in \T\setminus \J$. For simplicity we write $\W^{(j)}, \W^{(\J,j)}$ instead of $\W^{(\{j\})}, \W^{(\J \cup \{j\})}$. The same is applies to $\RR, m_n$ etc.

By $C$ and $c$ we denote some positive constants. 

For an arbitrary matrix $\A$ taking values in $\C^{n \times n}$ we define the operator norm by $\|\A\|: = \sup_{x \in \R^n: \|x\| = 1} \|\A x\|_2$, where $\|x\|_2 : = (\sum_{j = 1}^n |x_j|^2)^\frac12$. We also define the Hilbert-Schmidt norm by $\|\A\|_2: = \Tr^\frac12 \A \A^{*} = (\sum_{j,k = 1}^n |\A_{jk}|^2)^\frac12$.

\subsection{Main results} 
Without loss of generality we will assume in what follows that \(\X\) is a real symmetric matrix which satisfies the following conditions. 
\begin{definition}[Conditions \(\Cond\)] We say that a Hermitian random matrix \(\X\) satisfies conditions \(\Cond\) if its entries in the upper triangular part are independent random variables with \( \E X_{jk}^{(n)} = 0, \E |X_{jk}^{(n)}|^2 = 1\) and \( \max_{j,k,n} \E |X_{jk}^{(n)}|^4 =: \beta_4 < \infty\). 
\end{definition}
Our results proven below apply to the case of Hermitian matrices as well. Here we may additionally assume for simplicity that real and imaginary parts, \(\re X_{jk}, \imag X_{jk}\), are independent r.v. for all \(1 \le j < k \le n\). 
Otherwise  one needs to extend the moment inequalities for linear and quadratic forms in complex r.v. (see~\cite{GotzeNauTikh2015a}[Theorem A.1-A.2]) to the case of dependent real and imaginary parts, the details of which we omit.   

We will also often refer to the following condition \(\CondTwo\). 
\begin{definition}[Conditions \(\CondTwo\)]
	We say that the set of conditions \(\CondTwo\) holds if \(\Cond\) are satisfied and \(|X_{jk}| \le  \sqrt n/\overline R, 1 \le j,k \le n\), where \(\overline R \geq \log^3 n\).  
\end{definition}

Let us introduce the following notation
\begin{eqnarray*}
\Lambda_n(z) := m_n(z) - m_{sc}(z), \quad z = u + i v,
\end{eqnarray*}
where \(m_n(z), m_{sc}(z)\) were defined in~\eqref{m_n def} and~\eqref{convergence} respectively.  Recall that \(\imag \Lambda_n\) is the imaginary part of \(\Lambda_n\). The main result of this paper is the following theorem, which estimates the fluctuations~\eqref{fluctuations of m_n around s}.  
\begin{theorem}\label{th:main}
Assume that the conditions \(\CondTwo\) hold and let \(V > 0\) be some constant. 
\begin{itemize}
\item There exist positive constants \(A_0, A_1\) and \(C\) depending on \(V\) and \(\beta_4\)  such that
\begin{eqnarray}\label{abs fluct}
\E |\Lambda_n(u+iv)|^p \le \bigg(\frac{Cp}{nv}\bigg)^p,
\end{eqnarray}
for all \(1 \le p \le A_1 \log n\), \(V \geq v \geq A_0 n^{-1} \log^2 n\) and \(|u| \le 2+v\).
\item For any \(u_0 > 0\) there exist positive constants \(A_0, A_1\) and \(C\) depending on \(u_0, V\)
such that
\begin{eqnarray}\label{imag fluct}
\E |\imag \Lambda_n(u+iv)|^p \le \bigg(\frac{Cp}{nv}\bigg)^p,  
\end{eqnarray}
for all \(1 \le p \le A_1 \log n\), \(V \geq v \geq A_0 n^{-1} \log^2 n\) and \(|u| \le u_0\).
\end{itemize}
\end{theorem}

\begin{remark}
1. Using Markov's inequality the bound~\eqref{abs fluct} may be used to show that for any \(Q > 0\) there exists some positive constant \(C\) such that with probability at least \(1 -n^{-Q}\) for all \(v \geq A_0 n^{-1} \log^2 n\) and \(|u| \le 2 + v\):
\begin{eqnarray*}\label{fluctuations of m_n around s 2}
|\Lambda_n(u+iv)| \le \frac{C \log n}{nv}.
\end{eqnarray*}
Hence,~\eqref{fluctuations of m_n around s} holds with \(\alpha(n) \equiv 1\). 

\noindent 2. It is interesting to investigate the case of generalised matrix when for any \(j = 1, \ldots, n\) \(\sum_{k= 1}^n \E |X_{jk}|^2 = n\), but \(\E |X_{jk}|^2\) could be different. Unfortunately, the technique of the current paper doesn't allow to deal with such case directly. Fortunately, one may apply a combination of the {\it multiplicative} descent used in this paper (and first developed in~\cite{Schlein2014}) together with the {\it additive} descent developed in the series of papers by L.~Erd{\"o}s, B.~Schlein, H.-T.~Yau, J.~Yin et al; see e.g.~\cite{ErdosYauBOOK}. This combination was recently used in~\cite{GotzeNauTikh2017a}. We don't give details here to simplify the proof. 
\end{remark}

The result of the previous theorem may be formulated under conditions \(\Cond\). In this case one may truncate and re-normalize the entries of \(\X\) by means of Lemmas~\ref{appendix: lemma trunc 1}--~\ref{appendix: lemma trunc 2}  in the appendix. We obtain the following corollary. 
\begin{corollary}\label{maim cor}
Assume that the conditions \(\Cond\) hold and let \(V > 0\) be some constant. There exist positive constants \(A_0, A_1\) and \(C\) depending on \(V\) and \(\beta_4\)  such that
\begin{eqnarray*}
\E |\Lambda_n(u+iv)|^p \le \frac{C^p \log^{12p}n}{(nv)^p},
\end{eqnarray*}
for all \(1 \le p \le A_1 \log n\), \(V \geq v \geq A_0 n^{-1} \log^2 n\) and \(|u| \le 2+v\). Similar result holds true for~\eqref{imag fluct}. 
\end{corollary}
We believe that the power of the logarithm could be reduced. The main technical problem is in Lemmas~\ref{appendix: lemma trunc 1}--~\ref{appendix: lemma trunc 2}  in the appendix. Truncation on the level near \(\sqrt n\) requires additional logarithmic factors. 

\subsection{Sketch of the proof of Theorem~\ref{th:main}} \label{sketch}
To prove Theorem~\ref{th:main} we use the strategy from~\cite{GotzeNauTikh2016a}. 
\begin{enumerate}
\item Applying inequality~\eqref{eq: abs value lambda main result section} we may estimate \(\E|\Lambda_n|^p\) (depending on \(\re(z)\) being near or far from the spectral interval \([-2, 2]\)) by the moments \(\E |T_n(z)|^p\). This inequality first appeared in~\cite{Schlein2014}[Proposition 2.2]. 
\item Estimation of \(\E |T_n(z)|^p\) consists of two parts: \\
a) Estimation of \(\E|\RR_{jk}(z)|^p\); see Lemma~\ref{main lemma}. This bound requires to estimate high moments of \(\varepsilon_j\) (i.e. quadratic and linear forms \(\varepsilon_{2j}, \varepsilon_{3j}\)); see~\eqref{vareps def}. This step also uses the following crude bound
\begin{eqnarray}\label{T n crude bound sketch}
\E|T_n|^p \le \frac{1}{n}\sum_{j=1}^n \E^{1/2} |\varepsilon_j|^{2p}\E^{1/2} |\RR_{jj}|^{2p}. 
\end{eqnarray}
Unfortunately, the technique from~\cite{GotzeNauTikh2015a},~\cite{GotzeNauTikh2018TVP} doesn't work since we may truncate \(X_{jk}\) on the level \(\sqrt n / \overline R\), where \(\overline R\) is of the logarithmic order (opposite to the case when \(\E|X_{jk}|^{4+\delta} < \infty\). This allows to truncate on the level \(n^{1/2-\phi}\) for some small \(\phi\)). Let us demonstrate this on the quadratic form \(\varepsilon_{j2}\). Applying~\cite{GotzeNauTikh2018TVP}[Theorem~7] or~\cite{GotzeNauTikh2015a}[Theorem A.2] we obtain
\begin{eqnarray*}
\qquad\qquad\E |\varepsilon_{2j}^{(j)}|^p \le \frac{C^p p^p}{(nv)^{p/2}}\E \imag m_n^{(j)}(z) +
\frac{\beta_p p^{\frac{3p}{2} }}{n^p v^{p/2}} \sum_{k = 1}^n \E \imag^{p/2} \RR_{kk}^{(j)} + \frac{C^p p^{2p} \beta_{p}^2}{n^p} \sum_{k,l=1}^n \E |\RR_{kl}^{(j)}|^{2p}, 
\end{eqnarray*} 
where \(p\) satisfies \(1 \le p \le (nv)^{1/4}\). There is no problem to deal with first two terms in the r.h.s. of the previous inequality. The most difficult term is the last one. In the sub-Gaussian case this term has the order \(C^p p^{4p} (n^2v)^{-p}\) (see~\cite{GotzeNauTikh2016a}[Lemma 4.4] ) and is small for \(n^{-1} \le v \le 1\) (here we also use the fact that \( \beta_p^{1/p} \le C \sqrt p\)). Under assumptions \(\CondTwo\) we may only guaranteer that \(\E |\RR_{kl}^{(j)}|^{2p} \le C^p\). But \( \beta_p \le \beta_4 n^{p/2 - 2} \overline R^{4-p}, p \geq 4\). Hence, the last term in the estimate for  \(\varepsilon_{j2}\) is bounded by \(\beta_4^2 C^p p^{2p} \overline R^{-2p + 8}\), which could be very large for large \(p\). It is worth to mention here, that if one can truncate on the level, say, \(n^{1/4}\), then there will be an additional factor \(n^{p/2}\) in the denominator. 

To overcome this problem we use ideas from \cite{Aggarwal2016}. We introduce configuration matrix \(\LL:=[L_{jk}]_{j,k=1}^n\) such that \(L_{jk} = 1\) if \(|X_{jk}| \le n^{1/4} \underline R\) and \(L_{jk} = 0\) if \(n^{1/4} \underline R < |X_{jk}| \le n^{1/2} \overline R\) for some\(\underline R\) of the logarithmic order; see~\eqref{configuration matrix}. One may show that with high probability this matix
has the block structure (see~\eqref{configuration matrix 2}). This means that with high probability in each row and in each column of \(\X\) there is only small (of logarithmic order) number of {\it large} entries (\(L_{jk} = 0\)) and large number of {\it small} entries.  Fixing {\it admissible}  (see Definition~\ref{admissible} below) configuration \(\LL\) (corresponding to the block structure) one may estimate  \(\E(|\RR_{jk}(z)|^p \big | \LL), z = u + iv\); see Lemma~\ref{G_{jk} 0 step}. For each subrow with small entries we use bounds from~\cite{GotzeNauTikh2015a},~\cite{GotzeNauTikh2018TVP}. For each subrow with large entries we may use crude bounds which doesn't contain factor \(p^p\); see decomposition~\eqref{decomposition 1} and  corresponding estimates below.  Using now total probability rule and the crude bound \(\|\RR(z)\| \le v^{-1}\) if \(\LL\) is not admissible we  estimate \(\E|\RR_{jk}(z)|^p\).  \\

\noindent b) More accurate (than~\eqref{T n crude bound sketch}) bounds for \(\E |T_n(z)|^p\); see section~\ref{sec: bound for T}.  We use Lemma~\ref{lem: T_n general lemma} which provides a general framework for estimation of moments of statistics of independent r.v. This requires estimation of  \(\E |\varepsilon_j|^{\alpha}\) for \(1 \le \alpha \le 2\). The latter could be done since \(X_{jk}\) has moments of order up to \(4\). 
\end{enumerate}

\subsection{Applications of the main results} 
This section is addressed to application of Theorem~\ref{th:main} and Corollary~\ref{maim cor} to different questions as the rate of convergence of the ESD \(\mu_n\) to the semicircle law \(\mu_{sc}\), rigidity estimates for the eigenvalues \(\lambda_j, j = 1, \ldots, n\) and delocalization bounds for the corresponding eigenvectors \(\uu_j, j=1, \ldots, n\). Up to the power of logarithmic factors these results repeat the corresponding results from~\cite{GotzeNauTikh2015b},~\cite{GotzeNauTikh2016a}. We formulate all results with comments, but leave the proof. The interested reader may recover the proof from the corresponding papers mentioned above. It is worth to mention that these questions has been intensively studied under stronger assumptions in many papers; see e.g.~\cite{ErdosSchleinYau2009},~\cite{ErdosSchleinYau2009b},~\cite{ErdosSchleinYau2010},~\cite{ErdKnowYauYin2012},~\cite{ErdKnowYauYin2013a},~\cite{ErdKnowYauYin2013} and~\cite{TaoVu2011}. We also refer to the recent monograph~\cite{ErdosYauBOOK} and survey~\cite{TaoVu2011surv}.
\subsubsection{Rate of convergence of ESD}
Our first result provides quantitative estimates for the rate of convergence of the ESD to the semicircle law in the Kolmogorov distance. 
\begin{corollary} \label{th: rate of convergence}
Assume that the conditions \(\Cond\) hold. For any \(Q > 0\) there exists positive constant \(C\) such that with probability at least \(  1 - n^{-Q}\)
\begin{eqnarray*}
\Delta_n^{*}: = \sup_{x \in \R} |\mu_n((-\infty, x]) - \mu_{sc}((-\infty, x])| \le \frac{C \log^{12} n}{n}.
\end{eqnarray*}
\end{corollary}
For the proof see~\cite{GotzeNauTikh2016a}[Theorem 1.4]. The difference is in application of Corollary~\ref{maim cor} instead of~\cite{GotzeNauTikh2016a}[Theorem 1.1].  The proof is mainly based on application of the smoothing inequality (see e.g.~\cite{GotzeNauTikh2016a}[Corollary 6.2]) and Corollary~\ref{maim cor}. We believe that the power of the logarithm could be reduced from \(12p\) to \(p\) or even \(\frac{p}{2}\), which would be optimal  due to the result of Gustavsson~\cite{Gustavsson2005} for the  {\it Gaussian Unitary Ensembles} (GUE). 

Using this result on main prove the following corollary
\begin{corollary}
Assume that  conditions \(\Cond\) hold. For any \(Q > 0\) there exists positive constant \(C\) such that for all \(\Delta > 0\) with probability at least \(  1 - n^{-Q}\): 
\begin{eqnarray*}
\big| \#\big\{\lambda_j \in \big[x - \Delta/2n, x + \Delta/2n \big]\big\} -  g_{sc}(x) \, \Delta \big | \leq \frac{C \log^{12} n}{n }.
\end{eqnarray*}
\end{corollary}

\subsubsection{Rigidity}\label{rigidity sec}
Taking into account the result of Theorem~\ref{th: rate of convergence} and using Smirnov's transform one may also get the rigidity estimates for the majority of eigenvalues \( \lambda_j\). More precisely, one may control the eigenvalues on the bulk of the spectrum. To deal with the smallest (largest) eigenvalues one needs more accurate bound then in Theorem~\ref{th:main} for the distance between Stieltjes transforms.  For any \(u \in \R\) we define \(\kappa: = \kappa(u):= ||u|-2|\). 
\begin{theorem}\label{th: stronger bound for imag part}
Assume that the conditions \(\CondTwo\) hold and \(u_0 > 2\) and \(V > 0\). There exist positive constants \(A_0, A_1\) and \(C\) depending on \(u_0, V\) and \(\beta_4\) such that
\begin{eqnarray*}
\E|\imag \Lambda_n(z)|^p \le \frac{C^p p^p }{n^p (\kappa + v)^p} +  \frac{C^p p^{2p}}{(nv)^{2p} (\kappa+ v)^\frac{p}{2}}  + \frac{C^p p^\frac p2}{n^p v^\frac{p}{2} (\kappa + v)^\frac{p}{2}} + \frac{C^p p^p}{(nv)^\frac{3p}{2} (\kappa+ v)^\frac{p}{4} }.
\end{eqnarray*}
for all \(1 \le p \le A_1 \log n\), \(V \geq v \geq A_0 n^{-1} \log^2 n\) and \(2 \le |u| \le u_0\).
\end{theorem}

Let us define the  quantile  position of the \(j\)-th eigenvalue by
\begin{eqnarray*}
\gamma_j: \quad \int_{-\infty}^{\gamma_j} g_{sc}(\lambda) \, d\lambda = \frac{j}{n}, \quad 1 \le j \le n.
\end{eqnarray*}
The following results give the bounds for the fluctuations of \(\lambda_j\) around \(\gamma_j\). 
\begin{corollary}\label{th: rigidity}
Assume that the conditions \(\Cond\) hold and let \(K>0\) be an integer. Then
\begin{itemize}
\item (bulk) Let \(j \in [\log n, n - \log n+1]\) . For any \(Q > 0\) there exists positive constant \(C\) such that with probability at least \(1 - n^{-Q}\):
\begin{eqnarray*}
|\lambda_j - \gamma_j| \leq C_1 \log^{12} n [\min(j, n- j+1)]^{-\frac13} n^{-\frac23}. 
\end{eqnarray*}
\item (edge) Let \(j \le \log n\) or \( j \geq n - \log n + 1\). There exists positive constant \(C\) such that with probability at least \(1 -  \log^{-1} n \): 
\begin{eqnarray*}
|\lambda_j - \gamma_j| \le C \log^9 n [\min(j, n- j+1)]^{-\frac13} n^{-\frac23}.
\end{eqnarray*}
\end{itemize}
\end{corollary}
For the detailed proof see~\cite{GotzeNauTikh2015b}[Theorem 1.3] making minor changes. For the bulk of the spectrum we mainly use the following formula 
\begin{eqnarray*}
\lambda_j = G_{sc}^{-1}\left(\frac{j}{n} \right) + \E_\tau \frac{2\pi \theta \Delta_n^{*}}{\sqrt{4 - \left(G_{sc}^{-1}\left(\frac{j}{n} + \theta \Delta_n^{*} \right) \right)^2}};
\end{eqnarray*}
see proof of~\cite{GotzeNauTikh2015b}[Theorem 1.3]. Here, \(G_{sc}(x) := \mu_{sc}((-\infty, x])\). Taking into account that 
\begin{eqnarray}\label{eq: quantiles of semicircle law 1}
&&c_1 x^\frac23 \le 2+ G_{sc}^{-1}(x) \le c_2 x^\frac23 \quad \text{ for } x \in \big[0, 1/2 \big], \nonumber\\
&&c_1 (1-x)^\frac23 \le 2- G_{sc}^{-1}(x) \le c_2 (1-x)^\frac23 \quad \text{ for } x \in \big[1/2, 1\big], 
\end{eqnarray}
and Corollary~\ref{th: rate of convergence} one may obtain the estimates for the bulk of the spectrum. 
Clearly, the factor \(\log^{12} n\) comes from the bound for the \(\Delta_n^{*}\). The proof for the edge of the spectrum requires more involved technique. In particular, following~\cite{ErdKnowYauYin2013}[Theorem 7.6] we write
\[
\Pb(|\lambda_j - \gamma_j| \geq l) \le \Pb(|\lambda_j - \gamma_j| \geq l, \lambda_j > \gamma_j) + \Pb(|\lambda_j - \gamma_j| \geq l, \lambda_j < \gamma_j),
\]
where \(l: = C K j^{-1/3} n^{-2/3}\) for some \(C>0\).  The first case when \(\lambda_j > \gamma_j\) is trivial since in this situation \(\lambda_j > j_1 \geq -2 + c_1 n^{-2/3}\) (see~\eqref{eq: quantiles of semicircle law 1}) and we may repeat the calculations for the case of the bulk to get
$$
\Pb(|\lambda_j - \gamma_j| \geq l, \lambda_j > \gamma_j) \le n^{-Q}. 
$$
Applying~\eqref{eq: quantiles of semicircle law 1} we obtain $\gamma_j \le -2 + c_2\left(\frac{j}{n}\right)^\frac13 $. This enables to write the estimate
\begin{eqnarray}\label{operator norm} 
\Pb(|\lambda_j - \gamma_j| \geq l, \lambda_j < \gamma_j) \le \Pb \big(\lambda_1 \le -2 - (K/n)^{2/3}\big).
\end{eqnarray}
Estimation of the r.h.s. of the previous inequality requires to use truncation technique leading to very poor probability bounds (of order \(\log^{-1} n\)). Namely, we need to replace \(\W\) satisfying \(\Cond\) by the corresponding matrix \(\breve \W\) satisfying \(\CondTwo\). To estimate the r.h.s. of~\eqref{operator norm} one may follow~\cite{ErdKnowYauYin2013}[Theorem 7.3] and use  Theorem~\ref{th: stronger bound for imag part}.

\subsubsection{Delocalization of eigenvectors}
Let us denote by \(\uu_j := (u_{j 1}, ... , u_{j n})\) the eigenvectors of \(\W\) corresponding to the eigenvalue \(\lambda_j\). The following theorem is the direct corollary of Lemma~\ref{main lemma}.
\begin{corollary}\label{th: delocalization}
Assume that  conditions \(\Cond\) hold. There exist positive constant \(C\) such that with probability at least \( 1- \log^{-1} n\): 
\begin{eqnarray*}
\max_{1 \le j, k \le n} |u_{jk}| \leq C \sqrt{\frac{\log n}{n}}.
\end{eqnarray*}
\end{corollary}
Comparison with a similar result for the  {\it Gaussian Orthogonal Ensembles} (GOE) (see~\cite{AndersonZeit}[Corollary 2.5.4]) shows that this result is optimal with respect to the power of logarithm. For the proof see~\cite{GotzeNauTikh2016a}[Theorem 1.4], replacing~\cite{GotzeNauTikh2016a}[Lemma 3.1] with Lemma~\ref{main lemma}(with \(\alpha = 1\)). For the readers convenience we give an idea of the proof. We introduce the following distribution function
\[
F_{nj}(x): = \sum_{k=1}^n |u_{jk}|^2 \one[\lambda_k(\W)\le x].
\]
Using the eigenvalue decomposition of $\W$ it is easy to see that
\[
\RR_{jj}(z) = \sum_{k=1}^n\frac{|u_{jk}|^2}{\lambda_k(\W) - z} =  \int_{-\infty}^\infty \frac{1}{x - z} \, d F_{nj}(x).
\]
For any $\lambda > 0$ we have
\begin{eqnarray*}
\max_{1 \le k \le n} |u_{jk}|^2 \le \sup_x (F_{nj}(x + \lambda) - F_{nj}(x)) \le 2 \sup_u \lambda \imag \RR_{jj}(u + i\lambda).
\end{eqnarray*}
Estimation of the r.h.s. of the previous inequality requires again to use truncation technique leading to very poor probability bounds. 
Similarly to the edge case of Corollary~\ref{th: rigidity} we replace \(\W\) satisfying \(\Cond\) by the corresponding matrix \(\breve \W\) satisfying \(\CondTwo\), and apply Lemma~\ref{main lemma}(with \(\alpha = 1\)). Replacing conditions \(\Cond\) by \(\CondTwo\) one may improve the estimate. 

\section{Proof of the main result}\label{proof of the main result}
We start this section with the recursive representation of the diagonal entries of the resolvent \(\RR(z) := (\W - z \I)^{-1}\).  As noted before we shall systematically use for any matrix \(\W\) together with its resolvent \(\RR\), Stieltjes transform \(m_n\) and etc. the corresponding quantities \(\W^{(\J)}, \RR^{(\J)}, m_n^{(\J)}\) and etc. for the corresponding sub matrix with entries \(X_{jk}, j,k \in \T\setminus \J\). Here \(\T: = \{1, \ldots, n\}\) and \(\J \subset \T\). We will often omit the argument \(z\) from \(\RR(z)\) and write \(\RR\) instead. We may express \(\RR_{jj}\) in the following way
\begin{eqnarray}\label{eq: R_jj representation 0}
\RR_{jj} = \frac{1}{-z + \frac{X_{jj}}{\sqrt n} - \frac{1}{n}\sum_{l,k \in T_j} X_{jk} X_{jl} \RR_{kl}^{(j)}}.
\end{eqnarray}
Let \(\varepsilon_j : = \varepsilon_{1j} + \varepsilon_{2j}+\varepsilon_{3j}+\varepsilon_{4j}\), where
\begin{eqnarray}\label{vareps def}
\varepsilon_{1j} &:=&  \frac{1}{\sqrt n}X_{jj}, \quad \varepsilon_{2j} := -\frac{1}{n}\sum_{l \ne k \in T_j} X_{jk} X_{jl} \RR_{kl}^{(j)}, \quad \varepsilon_{3j} := -\frac{1}{n}\sum_{k \in T_j} (X_{jk}^2 -1) \RR_{kk}^{(j)}, \nonumber\\
\varepsilon_{4j}&:=& \frac{1}{n} (\Tr \RR - \Tr \RR^{(j)}).
\end{eqnarray}
Using these notations we may rewrite~\eqref{eq: R_jj representation 0} as follows
\begin{eqnarray}\label{eq: R_jj representation}
\RR_{jj} = -\frac{1}{z + m_n(z)} + \frac{1}{z + m_n(z)} \varepsilon_j \RR_{jj}.
\end{eqnarray}
Introduce
\begin{eqnarray}\label{eq:Lambda_b_bn}	
\Lambda_n: = m_n(z) - m_{sc}(z), \quad b(z): = z + 2 m_{sc}(z), \quad b_n(z) = b(z) + \Lambda_n,
\end{eqnarray}
and
\begin{eqnarray}\label{definition of T}
T_n: = \frac{1}{n} \sum_{j=1}^n \varepsilon_j \RR_{jj}.
\end{eqnarray}
Applying~\eqref{eq: R_jj representation}  we arrive at  the following representation for \(\Lambda_n\) in terms of \(T_n\) and \(b_n\)
\begin{eqnarray*}
\Lambda_n = \frac{T_n}{z + m_n(z) + m_{sc}(z)} = \frac{T_n}{b_n(z)}.
\end{eqnarray*}
It was proved in~\cite{Schlein2014}[Proposition 2.2] (see also~\cite{GotzeNauTikh2015a}[Lemma B.1]) that for all \(v > 0\) and \(|u| \le 2+v\)
(using the quantities \eqref{eq:Lambda_b_bn})
\begin{eqnarray}\label{eq: abs value lambda main result section}
|\Lambda_n| \le C \min\left\{\frac{|T_n|}{|b(z)|}, \sqrt{|T_n|}\right\}.
\end{eqnarray}
Moreover, for all \(v>0\) and \(|u| \le u_0\)
\begin{eqnarray}\label{eq: abs imag lambda main result section}
|\imag\Lambda_n| \le C \min\left\{\frac{|T_n|}{|b(z)|}, \sqrt{|T_n|}\right\}.
\end{eqnarray}
It is easy to check that \(b(z) = \sqrt{z^2-4}\) and moreover, there exist constants \(c, C>0\) such that \(c\sqrt{\kappa + v} \le |b(z)| \le C\sqrt{\kappa + v}\) for all \(|u| \le u_0\), \(0 < v \le V\), where \(\kappa\) is defined in the section~\ref{rigidity sec}; see e.g.~\cite{ErdKnowYauYin2013}[Lemma 4.3]. 
Hence, in order to bound \(\E |\Lambda_n|^p\) (or \(\E|\imag \Lambda|^p\) respectively) it is enough to control \(\E |T_n|^p\).

Let us introduce the following region in the complex plane:
\begin{eqnarray}\label{eq: definition reginon D}
\mathbb D_2: = \{z = u+iv \in \C: |u| \le u_0, V \geq v \geq v_0: = A_0 n^{-1} \log^2 n \},
\end{eqnarray}
where \(u_0, V\) are arbitrary fixed positive real numbers and \(A_0\) is some large constant defined below.

The following theorem provides a bound for \(\E |T_n|^p\) for all \(z \in \mathbb D_2\) in terms of diagonal resolvent entries. 
\begin{theorem}\label{th: general bound}
Assume that the conditions \(\CondTwo\) hold and \(u_0 > 2\) and \(V > 0\). There exist  positive constants \(A_0, A_1\) and \(C\)
depending on  \(u_0, V\) and \(\beta_4\) such that for all \(z \in \mathbb D_2\) we have
\begin{eqnarray}\label{eq: general bound}
\E|T_n|^p \le \frac{C^p p^p |b(z)|^{p}}{(nv)^p} +  \frac{C^p p^{2p}}{(nv)^{2p}},
\end{eqnarray}
where \(1 \le p \le A_1 \log n\).
\end{theorem}
\begin{remark}
To prove Theorem~\ref{th: stronger bound for imag part} one need more stronger bound for \(\E|T_n|^p\) than~\eqref{eq: general bound}. Minor changes in the proof of Theorem~\ref{th: general bound} will lead to the following estimate
\begin{eqnarray*}
\E|T_n|^p \le \frac{C^p p^p\mathcal A^p(4 p)}{(nv)^p}   +  \frac{ C^p p^{2p}}{(nv)^{2p}}  + \frac{C^p p^{p/2}|b(z)|^\frac{p}{2}\mathcal A^\frac{p}{2}(4 p)}{(nv)^p },
\end{eqnarray*}
where \(\mathcal A(q): = \max \big(\max_{j=1, \ldots, n} \E^{\frac{1}{q}} \imag^q \RR_{jj}, \imag m_{sc}(z) \big)\). This estimate is sufficient for our purposes. The term \( \mathcal A(q)\) may be estimate due to Lemma~\ref{eq: main lemma first statement}. We omit the details.  
\end{remark}
The proof of Theorem~\ref{th: general bound} is one of the crucial steps in the proof of the main result and will be given in the next section. We finish this section with the proof of Theorems~\ref{th:main} and~\ref{th: stronger bound for imag part}.

\begin{proof}[Proof of Theorem~\ref{th:main}]
To estimate \(\E|\imag \Lambda_n|^p\) we may choose one of the bounds~\eqref{eq: abs imag lambda main result section},
depending on whether \(z\) is near the edge of the spectrum or away from it. If \(|b(z)| \geq C p (nv)^{-1}\) then we may take the bound \( |\imag \Lambda_n| \le C |T_n|/|b(z)|\) and obtain
\begin{eqnarray*}
\E|\imag \Lambda_n|^p \le \frac{C^p \E |T_n|^p}{|b(z)|^p}\le \bigg(\frac{C p}{nv}\bigg)^p.
\end{eqnarray*}	
If the opposite inequality holds, \(|b(z)| < C p (nv)^{-1}\), then we will use the bound \( |\imag \Lambda_n| \le C \sqrt{|T_n|} \):
\begin{eqnarray*}
\E|\imag \Lambda_n|^p \le C^p \E^\frac12 |T_n|^{p} \le \bigg(\frac{C p}{nv}\bigg)^p.
\end{eqnarray*}
Both inequalities combined yield
\begin{eqnarray*}
\E|\imag \Lambda_n|^p \le \bigg(\frac{Cp}{nv}\bigg)^p.
\end{eqnarray*}
Similar arguments are applicable to  \(\E|\Lambda_n|^p\).
\end{proof}

\begin{proof}[Proof of Theorem~\ref{th: stronger bound for imag part}] Following the remark after Theorem~\ref{th: general bound} we may  conclude that
\begin{eqnarray*}\label{eq: T n 6 step}
\E|T_n|^p \le \frac{C^p p^p \imag^p m_{sc}(z)}{(nv)^p} +  \frac{C^p p^{2p}}{(nv)^{2p}}  + \frac{C^p |b(z)|^\frac{p}{2} \imag^\frac{p}{2}m_{sc}(z)}{(nv)^p } + \frac{C^p  p^{p/2}|b(z)|^\frac{p}{2} p^p}{(nv)^\frac{3p}{2} } .
\end{eqnarray*}
Using the bound  \( |\imag \Lambda_n| \le C |T_n|/|b(z)|\) we get
\begin{eqnarray*}\label{eq: Lambda_n stronger bound}
\E|\imag \Lambda_n|^p \le\frac{C^p p^p \imag^p m_{sc}(z)}{(nv)^p |b(z)|^p} +  \frac{C^p p^{2p}}{(nv)^{2p}|b(z)|^p}  + \frac{C^p p^\frac p2 \imag^\frac{p}{2}m_{sc}(z)}{(nv)^p |b(z)|^\frac{p}{2}} + \frac{C^p p^p}{(nv)^\frac{3p}{2} |b(z)|^\frac{p}{2} }.
\end{eqnarray*}
Since \(c\sqrt{\kappa + v} \le |b(z)| \le C\sqrt{\kappa + v}\) for all \(|u| \le u_0\), \(0 < v \le V\), and
\begin{eqnarray*}
\frac{c v}{\sqrt{\kappa + v}} \le \imag m_{sc}(z) \le \frac{c v}{\sqrt{\kappa + v}} \quad \text{ for all } 2 \le |u| \le u_0, 0 < v \le V,
\end{eqnarray*}
(see e.g.~\cite{ErdKnowYauYin2013}[Lemma 4.3]) we finally get
\begin{eqnarray*}\label{eq: Lambda_n stronger bound 2}
\E|\imag \Lambda_n|^p \le \frac{C^p p^p }{n^p (\kappa + v)^p} +  \frac{C^p p^{2p}}{(nv)^{2p} (\kappa + v)^\frac{p}{2}}  + \frac{C^p p^\frac p2}{n^p v^\frac{p}{2} (\kappa + v)^\frac{p}{2}} + \frac{C^p p^p}{(nv)^\frac{3p}{2} (\kappa+ v)^\frac{p}{4} }.
\end{eqnarray*}
This bound concludes the proof of the theorem.
\end{proof}

\section{Bounds for moments of diagonal entries of the resolvent}\label{sec: diagonal entries}
The main result of this section is the following lemma which provides a bound for moments of the diagonal entries of the resolvent.
Recall that (see the definition~\eqref{eq: definition reginon D}) for \(\alpha = 1, 2\)
\begin{eqnarray*}
\mathbb D_\alpha: = \{z = u+iv \in \C: |u| \le u_0, V \geq v \geq v_0: = A_0 n^{-1} \log^\alpha n \},
\end{eqnarray*}
where \(u_0, V > 0\) are any fixed real numbers and \(A_0\) is some large constant determined below. The first value \(\alpha = 1\) is sufficient to obtain optimal bounds for delocalization of eigenvectors. The second value, \(\alpha = 2\), is necessary for the main Theorem~\ref{th:main}. 
\begin{lemma}\label{main lemma}
Assuming the conditions \(\CondTwo\) there exist a positive constant \(H_0\) depending on \(u_0, V\) and positive constants \(A_0, A_1\) depending on \(H_0\) such that for all \(z \in \mathbb D_\alpha\) and \(1 \le p \le A_1 \log^\alpha n\)  we have
\begin{eqnarray}\label{eq: main lemma first statement}
&&\max_{1\le j \le k \le n} \E|\RR_{jk}(z)|^p \le H_0^p, \\
\label{eq: main lemma second statement}
&&\E \frac{1}{|z+m_n(z)|^p} \le H_0^p, \\
\label{eq: main lemma third statement}
&&\E \imag^p \RR_{jj}(z) \le H_0^p \bigg[ \imag^p m_{sc}(z) + \frac{ p^p}{(nv)^p} \bigg].
\end{eqnarray}
\end{lemma}
We provide the proof of~\eqref{eq: main lemma first statement} only. The proof of~\eqref{eq: main lemma second statement} and~\eqref{eq: main lemma third statement} is the same and will be omitted. For the details see~\cite{GotzeNauTikh2016a}[Lemma 3.1].

We start with introducing the following events
\begin{eqnarray*}
A_{jk}:=\{|X_{jk}|\le n^{\frac14}\underline R\},\quad B_{jk}:=A_{jk}^c = \{ n^\frac14 \underline R \le |X_{jk}| \le n^\frac12/\overline R \},
\end{eqnarray*}
where \(\underline R: = \underline R_n\) is some quantity depending on \(n\). We also denote 
\begin{eqnarray*}
p_n:= \Pb(n^\frac14\, \underline R \le |X_{jk}| \le n^\frac12 /\, \overline R)
\end{eqnarray*} 
Using Markov's inequality, it is easy to check that
\begin{eqnarray*}
p_n \le n^{-1} \underline R^4. 
\end{eqnarray*}
Following~\cite{Aggarwal2016} let us introduce the following configuration matrix
\begin{eqnarray*}
\LL:=\LL(\X):=[L_{jk}]_{j,k=1}^n
\end{eqnarray*}
with \(L_{jk}:=\one[A_{jk}]\). Let \(\xi_{jk}\) and \(\eta_{jk}\), \(j,k=1,\ldots,n\),  be mutually  independent random variables distributed  as \(X_{jk}\) conditioned on \(A_{jk}\) and \(B_{jk}\) resp.
Let \(\X(\LL):=[X_{jk}(L_{jk})]_{j,k=1}^n\)
where
\begin{eqnarray}\label{configuration matrix}
X_{jk}(L_{jk}):=\begin{cases}&\xi_{jk},\text{ if }L_{jk}=1,\\&\eta_{jk},\text{ if }L_{jk}=0.\end{cases}
\end{eqnarray}
We may consider matrix \(\W(\LL)\) as the matrix \(\W\) conditioned on the configuration \(\LL\).
We repeat some classification of configuration matrices from \cite{Aggarwal2016}.
\begin{definition}
{\rm Fix an \(n\times n\) configuration matrix \(\LL=[L_{jk}]_{j,k=1}^n\). We call} \(j,k\in \T\) {linked} {\rm (w.r.t \(\LL\)) if \(L_{jk}=0\); otherwise  we call them} {unlinked}.
\end{definition}
\begin{definition}{\rm  The indices \(j\) and \(k\) are} connected {\rm if there exists an integer \(l\) and a sequence of indices  \(j=j_1,j_2,\ldots,j_l=k\) such that \(j_{\nu}\) is} linked { \rm to \(j_{\nu+1}\) for each \(\nu\in[1,l-1]\).} 
\end{definition}
\begin{definition}
{\rm We call index \(j\)} deviant {\rm (w.r.t. to \(\LL\)) if there exist some index \(k\) such that \(j\) and \(k\) are linked. Otherwise \(j\) is called} typical. {\rm Let \(\mathcal D:=\mathcal D_{\LL}\subset \T\) denote the set of deviant indexes, and let \(\mathcal T:=\mathcal T_{\LL}\subset \T\) denote the set of typical indexes.}
\end{definition}

\begin{definition}
{\rm We say that \(\LL\) is:}
\begin{itemize} 
\item \quad {\it deviant-inadmissible} {\rm if there exist at least \( K \max(1, n^2 p_n) \) deviant indices, where \(K\) may depend on \(n\). }
\item  \quad {\it \(r\)-connected-inadmissible} {\rm if there exist distinct indices \(j_1, j_2, \ldots , j_r\) that are 
pairwise connected.} 
\end{itemize}
\end{definition}
For any \(\LL\) define 
\begin{eqnarray*}
r(\mathbf L):=\max\{r\ge 1:\ \LL \text{ is } \text{\(r\)-connected-inadmissible}\}.
\end{eqnarray*}
\begin{definition}\label{admissible}
We call configuration \(\LL\) \(r\)-admissible, if it is not deviant-inadmissible and
\begin{eqnarray*}
r(\mathbf L)\le r-1.
\end{eqnarray*}
\end{definition}
Following~\cite{Aggarwal2016} we may estimate   
\begin{eqnarray*}
\Pb ( \LL \text{ is \(r\)-connected-inadmissible} ) &\le& \binom{n}{r}  \binom{r^2}{r-1} p_n^{r-1} \\&\le& \bigg( \frac{ne}{r} \bigg)^r \bigg( \frac{e r^2}{r-1} \bigg)^{r-1} \bigg( \frac{1}{n\underline R^4} \bigg)^{r-1} \le  e^{3r} n r^{-1} \underline R^{-4(r-1)}.
\end{eqnarray*}
Applying Chernoffs inequality we may also show that
\begin{eqnarray*}
\Pb ( \LL \text{ is deviant-inadmissible} ) \le e^{- c K \max(1, n^2 p_n)}.
\end{eqnarray*}
Denote by \( \mathcal L_r\) the set of all \(\LL\) \(r\)-admissible configurations. In what follows we take \( r := \log^3 n, \underline R := \log n, K:= \log^3 n\). Then
\begin{eqnarray}\label{bad conf}
\Pb( \LL \notin \mathcal L_r) \le n^{-c \log^2 n}
\end{eqnarray}
for some large \(c > 0 \). 

Let us fix the \(r\)-admissible configuration \(\LL\). By definition of \(r\)-admissibility we may find Hermitian matrices \(\A_{\nu}\) of order \(r_{\nu}\le r\), \(\nu=1,\ldots,L\)  such that \(r_1+\ldots+r_L \le K \max (1, n^2 p_n)\) and matrix \(\LL\) may be rewritten as follows
\begin{eqnarray}\label{configuration matrix 2}
\LL=\begin{bmatrix}&\A_1&1\ldots&1\ldots&1\ldots\\
&1\ldots &\A_2&\ldots&1\ldots\\
&\ldots&\ldots&\ldots\\
&1\ldots&1\ldots &\A_L& 1 \ldots \\
&1\ldots& 1\ldots&\ldots& 1 \ldots\\
&\ldots&\ldots&\ldots\\
&1\ldots&1\ldots&\ldots& 1 \ldots
\end{bmatrix}.
\end{eqnarray}
Moreover, the zero-entries of matrix \(\LL\) can only be inside of \(\A_{\nu}\), and in each row (column) may contain at most \(r\) zero-entries.

Denote \(\HH := \X(\LL)\) and \(\G := (n^{-1/2}\HH - z\I)^{-1}\).  We also assume that \( \HH = [h_{jk}]_{j,k=1}^n\).   
\begin{lemma}\label{G_{jk} 0 step}
Let \(\LL\) be \(r\)-admissible. Assuming the conditions \(\CondTwo\) there exist a positive constant \(C_0\) depending on \(u_0, V\) and positive constants \(A_0, A_1\) depending on \(C_0\) such that for all \(z \in \mathbb D_\alpha\) and \(1 \le p \le A_1(nv)^{1/4}/\overline R\) we have
\begin{eqnarray*}\label{eq: main lemma first statement 0}
\max_{1\le j,k \le n} \E|\G_{jk}|^p \le C_0^p.
\end{eqnarray*}
\end{lemma}

Since \(u\) is fixed and \(|u| \le u_0\) we shall omit \(u\) from the notation of the resolvent and denote \(\G(v): = \G(z)\). Sometimes in order to simplify notations we shall also omit the  argument \(v\) in \(\G(v)\) and just write \(\G\). For any \(j \in \T_{\J}\) (see section~\ref{sec: notation}) we may express \(\G_{jj}^{(\J)}\) in the following way (compare with~\eqref{eq: R_jj representation 0})
\begin{eqnarray*}\label{eq: representation for RR_jj}
\G_{jj}^{(\J)} = - \frac{1}{z + m_n^{(\J)}(z) -  \ve_j^{(\J)}  },
\end{eqnarray*}
where \( m_n^{(\J)}(z): = \frac{1}{n} \Tr \G^{(\J)}(z)\) and  \(\ve_j^{(\J)} : = \ve_{1j}^{(\J)} + \ldots +\ve_{5j}^{(\J)}\). Here
\begin{eqnarray*}
\ve_{1j}^{(\J)} &:=&  \frac{1}{\sqrt n} h_{jj}, \quad \varepsilon_{2j}^{(\J)} := -\frac{1}{n}\sum_{l \ne k \in \T_{\J,j}} h_{jk} h_{jl} \G_{kl}^{(\J,j)},
\quad \ve_{3j}^{(\J)} := -\frac{1}{n}\sum_{k \in T_{\J,j}} (h_{jk}^2 -\sigma_{jk}^2) \G_{kk}^{(\J, j)}, \\
\ve_{4j}^{(\J)}&:=& \frac{1}{n} (\Tr \G^{(\J)} - \Tr \G^{(\J,j)}), \quad \ve_{5j}^{(\J)} := \frac{1}{n}\sum_{k \in T_{\J,j}} (1 - \sigma_{jk}^2) \G_{kk}^{(\J, j)}.
\end{eqnarray*}
We also introduce the quantities \(\Lambda_n^{(\J)}(z) : = m_n^{(\J)} (z) - m_{sc}(z)\) and
\begin{eqnarray*}
T_n^{(\J)}: = \frac{1}{n} \sum_{j \in \T_{\J}} \varepsilon_j^{(\J)}\G_{jj}^{(\J)}.
\end{eqnarray*}
The following lemma allows  to  recursively estimate the moments of \(\G_{jj}^{(\J)}\). 
\begin{lemma}\label{lemma: recurence relation for R_jj}
For an arbitrary set \(\J \subset \T\) and all \(j \in \T_\J\) there exist a positive constant \(c_0\) depending on \(u_0, V\) only such that for all \(z = u + i v\) with \(V \geq v > 0\) and \(|u| \le u_0\) we have
\begin{eqnarray*}\label{inequality for R_jj}
|\G_{jj}^{(\J)}| \le c_0\Big(1 + |T_n^{(\J)}|^{\frac{1}{2}}|\G_{jj}^{(\J)}| + |\varepsilon_j^{(\J)}||\G_{jj}^{(\J)}|\Big).
\end{eqnarray*}
\end{lemma}
\begin{proof} 
The proof may be found in~\cite{Schlein2014}[Lemma 3.4]  or~\cite{GotzeNauTikh2015a}[Lemma~4.2].
\end{proof}
Let us take \(\tilde v_0 := A_0 n^{-1} \log^{4(\alpha+1)} n\). 
\begin{lemma}\label{lemma step for resolvent} Let \(\LL\) be \(r\)-admissible and assume that the conditions \(\CondTwo\) hold. Let \(C_0\) and \(s_0\) be arbitrary numbers such that \(C_0 \geq \max(1/V, 6 c_0), s_0 \geq 4\). There exist a sufficiently large constant \(A_0\) and small constant \(A_1\) depending on \(C_0, s_0, V\) only such that the following statement holds. Fix some \(\tilde v: \tilde v_0 s_0 \leq \tilde v \le V\). Suppose that for some integer \(L > 0\), all \(u, v',q\) such that  \(\tilde v \leq v' \leq V,\, |u| \le u_0, 1 \le q \le A_1 (n v')^{1/4}/\underline R\)
\begin{eqnarray}\label{main condition 1 0}
\max_{\J: |\J| \le L} \max_{l, k \in \T_\J}\E |\G_{l k}^{(\J)}(v')|^q \le C_0^q.
\end{eqnarray}
Then for all \(u,v, q\) such that \(\tilde v/s_0 \leq v \le V, |u| \le u_0\), \(1 \le q \le A_1 (n v)^{1/4}/\underline R\)
\begin{eqnarray*}\label{main condition 1 1}
\max_{\J: |\J| \le L-1} \max_{l, k \in \T_\J}\E |\G_{l k}^{(\J)}(v)|^q \le C_0^q.
\end{eqnarray*}
\end{lemma}
\begin{proof}
Let us fix an arbitrary \(s_0 \geq 4\) and \(v \geq \tilde v/s_0\), \(\J \subset \T\) such that \(|\J| \le L-1\). In the following let \(j, k \in \T_\J\). 
First we note that for any \( j = 1, \ldots, n\) and \( 1 \le q \le A_1 (nv) \) we may write
\begin{eqnarray}\label{G diag est}
\E|\G_{jj}^{(\J,j)}|^{2q} \le s_0^{2q} \E|\G_{jj}^{(\J,j)}(s_0 v)|^{2q} \le (s_0 C_0)^{2q} 
\end{eqnarray}
(the same inequalities hold for \(\G^{(\J,j)}\) replaced by \( \G^{(\J)}\)). The first inequality follows from the fact that \(\G(v) \le s_0 \G(s_0v)\) (see e.g~\cite{Schlein2014}[Lemma 3.4] or~\cite{GotzeNauTikh2015a}[Lemma C.1]) and the second inequality follows from the assumption~\eqref{main condition 1 0}. 
Moreover, for any \(1 \le j<k \le n\) we have
\begin{eqnarray}\label{ G off diag est}
\E|\G_{jk}^{(\J,j)}(v)|^{2q} &\le& 2^{4q}\E|\G_{jk}^{(\J,j)}(s_0 v)|^{2q} + 2^{2q} s_0^{4q} \E^\frac12|\G_{jj}^{(\J,j)}(s_0 v)|^{2q} \E^\frac12|\G_{kk}^{(\J,j)}(s_0 v)|^{2q} \nonumber  \\
&\le& (3 s_0 C_0)^{2q}.
\end{eqnarray}
Applying Lemma~\ref{lemma: recurence relation for R_jj} we get
\begin{eqnarray*}
\E|\G_{jj}^{(\J)}|^q &\le& (3c_0)^q\Big(1 + \E^\frac{1}{2}|T_n^{(\J)}|^{q}\E^\frac12|\G_{jj}^{(\J)}|^q + \E^\frac12|\varepsilon_j^{(\J)}|^{2q}\E^\frac12|\G_{jj}^{(\J)}|^{2q}\Big) \\
&\le& (3c_0)^q\Big(1 + (s_0 C_0)^\frac q2 \E^\frac{1}{2}|T_n^{(\J)}|^{q} + (s_0 C_0)^q \E^\frac12|\varepsilon_j^{(\J)}|^{2q}\Big).
\end{eqnarray*}
By the Cauchy-Schwarz inequality 
\begin{eqnarray*}\label{eq: T_n est}
\E |T_n^{(\J)}|^{q} &\le&  \left(\frac{1}{n} \sum_{j \in \T_{\J}} \E |\varepsilon_j^{(\J)}|^{2q} \right)^{1/2} \left(\frac{1}{n} \sum_{j \in \T_{\J}} \E |\G_{jj}^{(\J)}(v)|^{2q} \right)^{1/2} \le (s_0C_0)^q \max_{j \in \T_{\J}} \E^{1/2} |\varepsilon_j^{(\J)}|^{2q}.
\end{eqnarray*}
The last two inequalities imply that
\begin{eqnarray*}
\E|\G_{jj}^{(\J)}|^q \le (3c_0)^q \big(1 + (s_0 C_0^q) \big(\max_{j \in \T_{\J}} \E^{1/4} |\varepsilon_j^{(\J)}|^{2q} +  \E^\frac12|\varepsilon_j^{(\J)}|^{2q}\big)\big).
\end{eqnarray*}
It remains to estimate \(\E|\varepsilon_j^{(\J)}|^{2q}\). By an obvious inequality we have
\begin{eqnarray*}
\E |\varepsilon_{j}^{(\J)}|^{2q} \le 5^{2q} (\E |\varepsilon_{1j}^{(\J)}|^{2q} + \ldots +  \E|\varepsilon_{5j}^{(\J)}|^{2q}). 
\end{eqnarray*}	
Let \( \mathbb A_j: = \{ k \in \T: L_{jk} = 0 \} \). For a \(r\)-admissible configuration \(|\mathbb A_j| \le r\). 
It is easy to check that 
\begin{eqnarray}\label{eq: h_jk est}
\E|h_{jk}|^q \le \begin{cases}
\mu_4 \underline R^{q-4} n^{q/4-1}, &\text{ if } k \notin \mathbb A_j\\
n^{q/2} \overline R^{-q}, & \text { if } k \in \mathbb A_j.
\end{cases}
\end{eqnarray}
Moreover, for \(k \notin \mathbb A_j\) 
\begin{eqnarray}
|\E h_{jk}| \le \frac{\mu_4}{\underline R^3 n^{3/4}}.
\end{eqnarray}
The bound for \( \E|\ve_{1q}^{(\J)}|^{2q} \) is the direct corollary of~\eqref{eq: h_jk est}
\begin{eqnarray}\label{ve1}
\E|\ve_{1q}^{(\J)}|^{2q} \le \mu_4 \underline R^{2q-4} n^{-q/2-1} + \overline R^{-2q}.  
\end{eqnarray} 
Let us consider \(\ve_{2j}^{(\J)}\). We may rewrite it as a sum
\(\ve_{2j}^{(\J)} = \zeta_{1j} + \zeta_{2j} + \ldots + \zeta_{6j}\),
where
\begin{eqnarray}
\label{decomposition 1}
	\zeta_{1j}&:=& -\frac{1}{n}\sum_{l \ne k \in \T_{\J,j} \cap \mathbb A_j} \eta_{jk} \eta_{jl} \G_{kl}^{(\J, j)}, \nonumber\\
	\zeta_{2j}&:=& -\frac{1}{n}\sum_{l \ne k \in \T_{\J,j} \setminus \mathbb A_j} (\xi_{jk} - \E \xi_{jk}) (\xi_{jl} - \E \xi_{jl}) \G_{kl}^{(\J, j)}, \nonumber\\
	\zeta_{3j}&:=& -\frac{2}{n}\sum_{l \ne k \in \T_{\J,j} \setminus \mathbb A_j}  (\xi_{jl} - \E \xi_{jl}) (\E \xi_{jk}) \G_{kl}^{(\J, j)}, \nonumber\\
	\zeta_{4j}&:=& -\frac{2}{n}\sum_{l \ne k \in \T_{\J,j} \setminus \mathbb A_j}  (\E \xi_{jl}) (\E \xi_{jk}) \G_{kl}^{(\J, j)}, \nonumber\\
	\zeta_{5j}&:=& -\frac{2}{n}\sum_{l \in \T_{\J,j} \cap \mathbb A_j} \eta_{jl} \sum_{k \in \T_{\J, j,l}\setminus \mathbb A_j} (\xi_{jk} - \E \xi_{jk}) \G_{kl}^{(\J, j)}, \nonumber\\
	\zeta_{6j}&:=& -\frac{2}{n}\sum_{l \in \T_{\J,j} \cap \mathbb A_j} \eta_{jl} \sum_{k \in \T_{\J, j,l}\setminus \mathbb A_j} (\E\xi_{jk}) \G_{kl}^{(\J, j)}.
\end{eqnarray}
Applying a crude bound and~\eqref{ G off diag est} we get
\begin{eqnarray*}
	\E|\zeta_{1j}|^{2q} \le \frac{r^{4q-2}}{n^{2q}} \sum_{l \ne k \in T_{\J,j} \cap \mathbb A_j} \E |\G_{kl}^{(\J, j)}|^{2q} \E|\eta_{jk} \eta_{jl}|^{2q}  \le  \frac{C^q r^{4q} (s_0 C_0)^{2q}}{ \overline R^{4q}}.
\end{eqnarray*}
Using Burkholder's type inequality and~\eqref{eq: h_jk est} we obtain the following estimate 
\begin{eqnarray*}
\E |\zeta_{2j}|^{2q} &\le& \frac{C^q q^q \E \imag^q m_n^{(\J, j)}}{(nv)^q} + \frac{C^q q^{3q} \underline R^{2q-4}}{(n^{3/2} v)^q } \sum_{l \in \T_{\J,j} \setminus \mathbb A_j} \E \imag^q \G_{ll}^{(\J,j)} \\
&+&\frac{C^q q^{4q} \underline R^{4q-8}}{n^2 n^q} \sum_{l \ne k \in \T_{\J,j} \setminus \mathbb A_j}  \E|\G_{kl}^{(\J,j)}|^{2q}.
\end{eqnarray*}
This inequality and ~\eqref{G diag est}--\eqref{ G off diag est}  together imply
\begin{eqnarray}
\E |\zeta_{2j}|^{2q} \le \frac{C^q q^q (s_0 C_0)^q }{(nv)^q} + \frac{C^q q^{3q} \underline R^{2q-4} (s_0 C_0)^q}{ (n^{3/2} v)^q }  + \frac{C^q q^{4q} \underline R^{4q-8} (s_0 C_0)^{2q}}{ n^q} .
\end{eqnarray}
For the term \(\zeta_{3j}\) we use the Rosenthal inequality, see e.g.~\cite{Rosenthal1970}, 
\begin{eqnarray}
\E|\zeta_{3j}|^{2q} &\le& \frac{C^q q^{q} (s_0 C_0)^q}{(n^{3/2}v)^q \underline R^{6q}}.
\end{eqnarray}
Again the crude bound imply 
\begin{eqnarray}
\E |\zeta_{4j}|^{2q} \le \frac{C^q ( s_0 C_0)^{2q} }{(n^2v)^{q} \underline R^{12q}}.
\end{eqnarray}
Similarly, 
\begin{eqnarray}
\E|\zeta_{5j}|^{2q} &\le& \frac{C^q r^{2q-1}}{n^{2q}} \sum_{l \in \T_{\J,j} \cap \mathbb A_j} \E|\eta_{jl}|^{2q} \E \bigg| \sum_{k \in \T_{\J, j,l}\setminus \mathbb A_j} (\xi_{jk} - \E \xi_{jk}) \G_{kl}^{(\J,j)} \bigg|^{2q} \nonumber \\
&\le& \frac{C^q r^{2q} q^{q} (s_0 C_0)^q }{(nv)^{q} \overline R^{2q}} + \frac{C^q r^{2q}  \underline R^{2q-4}q^{2q} ( s_0C_0)^{2q}}{ n^{q/2} \overline R^{2q}}. 
\end{eqnarray}
Finally,
\begin{eqnarray}
\E|\zeta_{6j}|^{2q} \le \frac{C^q r^{2q} (s_0C_0)^{2q}}{ (n^{3/2}v)^q \overline R^{2q} \underline R^{6q}}.
\end{eqnarray}
For the term \(\ve_{3j}^{(\J)}\) we may proceed similarly. We get that \( \ve_{3j}^{(\J)} = \widehat \zeta_{1j} +  \zeta_{2j}\), where
\begin{eqnarray*}
\widehat \zeta_{1j}&:=& -\frac{1}{n}\sum_{k \in T_{\J,j} \cap \mathbb A_j} (\eta_{jk}^2 -\sigma_{jk}^2) \G_{kk}^{(\J, j)}, \\
\widehat \zeta_{2j}&:=& -\frac{1}{n}\sum_{k \in T_{\J,j} \setminus \mathbb A_j} (\xi_{jk}^2 -\sigma_{jk}^2) \G_{kk}^{(\J, j)}.
\end{eqnarray*}
The crude bound implies that
\begin{eqnarray}
\E|\widehat \zeta_{1j}|^{2q} \le \frac{C^q r^{2q} (s_0 C_0)^{2q}}{ \overline R^{4q}}. 
\end{eqnarray}
Applying the Rosenthal inequality we get
\begin{eqnarray}
\E |\widehat \zeta_{2j}|^{2q} &\le& \frac{C^q q^{q}}{n^{2q}} \bigg(\sum_{k \in T_{\J, j} \setminus \mathbb A_j} |\G_{kk}^{(\J, j)}|^2 \bigg)^{q} + \frac{C^q q^{2q} \underline R^{4q-4}}{n^{q+1}} \sum_{k \in T_{\J, j} \setminus \mathbb A_j} |\G_{kk}^{(\J,j)}|^{4q} \nonumber \\
&\le& \frac{C^q q^{q} (s_0 C_0)^{2q}}{n^{q}}  + \frac{C^q q^{2q} \underline R^{4q-4} (s_0 C_0)^{4q}}{n^{q}}.
\end{eqnarray}
It is straightforward to check that
\begin{eqnarray}
	\E|\ve_{4j}^{(\J)}|^{2q} \le \frac{1}{(nv)^{2q}}. 
\end{eqnarray}
Finally, for \(\ve_{5j}^{(\J)}\) we may write
\begin{eqnarray}
\E |\ve_{5j}^{(\J)}|^{2q} \le \frac{C^{2q} r^{2q} (s_0 C_0)^{2q}}{\overline R^{4q}}  + \frac{C^q (s_0 C_0)^{2q}}{n^{q}\underline R^{2q}} .
\end{eqnarray}
The off-diagonal entries \(\G_{jk}^{(\J)}\) may be expressed as follows
\begin{eqnarray*}
\G_{jk}^{(\J)} = -\frac{1}{\sqrt n} \G_{jj}^{(\J)} \sum_{l \in \T_{\J,j}} h_{jl} \G_{lk}^{(\J,j)} .
\end{eqnarray*}
Applying  
\begin{eqnarray*}
\E|\G_{jk}^{(\J)}|^q \le \frac{(s_0 C_0)^q }{n^{q/2}} \E^\frac{1}{2} \bigg| \sum_{l \in \T_{\J,j}} h_{jl} \G_{lk}^{(\J,j)} \bigg|^{2q}.
\end{eqnarray*}
We proceed similarly to the estimation of \(\E|\varepsilon_j^{(\J)}|^{2q}\). For simplicity let us denote
\begin{eqnarray*}
\overline \zeta_{1j}&:=& \frac{1}{\sqrt n} \sum_{l \in \T_{\J, j} \cap \mathbb A_j} \eta_{jk} \G_{lk}^{(\J,j)}, \\
\overline \zeta_{2j}&:=& \frac{1}{\sqrt n}\sum_{l \in \T_{\J, j} \setminus \mathbb A_j} (\xi_{jk} - \E \xi_{jk}) \G_{lk}^{(\J,j)}, \\
\overline \zeta_{3j}&:=& \frac{1}{\sqrt n}\sum_{l \in \T_{\J, j} \setminus \mathbb A_j} (\E \xi_{jk}) \G_{lk}^{(\J,j)}.
\end{eqnarray*}
The crude bound implies that
\begin{eqnarray}
\E|\overline \zeta_{1j}|^{2q} \le \frac{ r^{2q} (s_0 C_0)^{2q} }{ \overline R^{2q}}.
\end{eqnarray}
By Rosenthal's inequality
\begin{eqnarray}
\E|\overline \zeta_{2j}|^{2q} \le \frac{C^q q^{q/2} s_0^q C_0^q }{(nv)^q } + \frac{C^q q^q \underline R^{2q - 4} (s_0 C_0)^{2q}}{n^{q/2}}.
\end{eqnarray}
Finally,
\begin{eqnarray}\label{zeta 3}
\E|\overline \zeta_{j3}|^{2q} \le \frac{C^q (s_0 C_0)^{2q} }{\overline R^{6q} (n^{3/2} v)^{q}}.
\end{eqnarray}

Analysing~\eqref{ve1}--\eqref{zeta 3} it is easy to see that one may choose sufficiently large constant \(A_0\) and small constant \(A_1\) such that 
\begin{eqnarray*}
\E |\G_{jk}|^{q} \le C_0^q. 
\end{eqnarray*}
\end{proof}

\begin{proof}[Proof of Lemma~\ref{G_{jk} 0 step}] Let us choose some sufficiently large constant \(C_0 > \max(1/V, 6 c_0)\), where \(c_0\) is defined in Lemma~\ref{lemma: recurence relation for R_jj}. We also choose \(A_0, A_1\) as in Lemma~\ref{lemma step for resolvent} \(s_{0}: = 2^4\). Let \(L:= [\log_{s_{0}} \tilde V/\tilde v_0]+1\). Since \(\|\G^{(\J)}(V)\| \le V^{-1}\) we may write 
\begin{eqnarray*}
\max_{\J: |\J| \le  L} \max_{l,k \in \T_\J} \E|\G_{lk}^{(\J)}(V)|^p \le C_0^p
\end{eqnarray*}
for all \(u, p\) such that \(|u| < 2\) and \(1 \le p \le A_1 (nv)^{1/4}/\underline R\). Fix arbitrary \(v: V/s_{0} \le  v \leq V\) and \(p: 1 \le p \le (nv/s_0)^{1/4}/\underline R\). Lemma~\ref{lemma step for resolvent} yields that
\begin{eqnarray*}
\max_{\J: |\J| \le L-1} \max_{l, k \in \T_\J} \E|\G_{lk}^{(\J)}(v)|^p \le C_0^p
\end{eqnarray*}
for \(1 \le p \le A_1 (n V/s_0)^{1/4}/\underline R\), \(v \geq V/s_0\). We may repeat this procedure \(L\) times and finally obtain
\begin{eqnarray*}
\max_{l,k \in \T}\E|\G_{lk}(v)|^p \le C_0^p
\end{eqnarray*}
for \(1 \le p \le A_1 (n V /s_0^{L})^{1/4}/\underline R \le A_1 (n \tilde v_0)^{1/4}/\underline R\) and \(v \geq V/s_0^{L} = \tilde v_0\). 
\end{proof}

The previous lemma allows to obtain \(p = A_1 \log^\alpha n\) by taking \(v = \tilde v_0 = A_0 n^{-1} \log^{4(\alpha+1)} n\). Without loss of generality we may consider \(p = A_1 \log^\alpha n\) only (otherwise one may apply Lyapunov's inequality for moments). It follows from Lemma~\ref{G_{jk} 0 step} that for any \(r\)-admissible \(\LL\):
\begin{eqnarray*}
\max_{j,k \in \T} \E |\G_{jk}(v)|^p \le C_0^p
\end{eqnarray*}
for all \(V \geq v \geq \tilde v_0\). We may descent from \(\tilde v_0\) to \(v_0\) while keeping \(p = A_1 \log^\alpha n\).   Indeed, first we may take \(s: = \log^{3\alpha-1} n\) and show that for \(v \geq v_0\)
\begin{eqnarray*}
\max_{j,k \in \T} \E |\G_{jk}(v)|^p \le C_0^p \log^{(3\alpha-1) p} n.
\end{eqnarray*}
It remains to remove the log factor from the r.h.s. of the previous equation. To this aim we shall adopt the moment matching technique which has been successfully  used  recently in~\cite{LeeYin2014} and~\cite{GotzeNauTikh2016a}. 

We consider the pairs \( (j,k): L_{jk} = 1\), and denote by \(\overline \xi_{jk}\) random variables such that: \(|\overline \xi_{jk}| \le D\), for some \(D\) chosen later, and
\begin{eqnarray*}
\E \xi_{jk}^s = \E \overline \xi_{jk}^s \, \text { for } \, s = 1, ... , 4.
\end{eqnarray*}
It follows from~\cite{LeeYin2014}[Lemma~5.2]. that such a set of random variables exists. Let us denote \(\HH^\y: = [h_{jk}^\y]_{j,k=1}^n\) such that
\begin{eqnarray}\label{sub gaussian matrix}
h_{jk}^\y = \begin{cases}
\overline \xi_{jk}, & \text{ if } L_{jk} = 1, \\
\eta_{jk}, & \text{ otherwise.}
\end{cases}
\end{eqnarray}
Introduce \(\G^\y: = (n^{-1/2}\HH^\y - z \I)^{-1}\) and \(m_n^\y(z): = \frac{1}{n} \Tr\G^\y(z)\). Then, in Lemma~\ref{lem: bound in the optimal region} we show that for all \(v \geq v_0\)  and \(5 \le p \le A_1 \log^\alpha n\) there exist positive constants \(C_1, C_2\) such that
\begin{eqnarray*} \label{bound via sub-gaussian case}
\E|\G_{jk}(v)|^p \le C_1^p + C_2 \E|\G_{jk}^\y(v)|^p.
\end{eqnarray*}
It is easy to see that \(\overline \xi_{jk}\) are sub-Gaussian random variables. Repeating the proof of Lemma~\ref{lemma step for resolvent}, see Lemma~\ref{lemma step for resolvent subgauss} in the appendix, we get 
\begin{eqnarray*}\label{sub gaussian case}
\E|\G_{jk}^\y(v)|^p \le H_0^p
\end{eqnarray*}
for some \(H_0 > C_0\).  We omit the details and proceed to the proof of Lemma~\ref{main lemma}.

\begin{proof}[Proof of Lemma~\ref{main lemma}] 
From~\eqref{bad conf} we conclude that
\begin{eqnarray*}
\Pb( \LL \notin \mathcal L_r) \le n^{-c \log^2 n}
\end{eqnarray*}
for some large \(c > A_1\). It is easy to see that 
\begin{eqnarray*}
\E |\RR_{jk}(v)|^p &=& \E \big[ \one[\LL  \in \mathcal L_r] \E(|\RR_{jk}(v)|^p\big| \LL)\big] + \E \big[ \one[\LL  \notin \mathcal L_r] \E(|\RR_{jk}(v)|^p\big| \LL)\big] \\
&\le& H_0^p + \Pb^\frac12 (\LL  \notin \mathcal L_r) \E^\frac12|\RR_{jk}(v)|^{2p} \le H_0^p + v^{-p} \Pb^\frac12 (\LL  \notin \mathcal L_r) \\
&\le& H_0^p + n^{A_1 \log^ \alpha n - c \log^2 n  - \log^{\alpha-1} n \log \log n}  \le H_1^p,  
\end{eqnarray*}
for some \(H_1 > H_0\).
\end{proof}

\section{Estimate of \(T_n\)} \label{sec: bound for T}
In this section we prove Theorem~\ref{th: general bound}.  We will follow the main idea of the proof of corresponding results in~\cite{GotzeNauTikh2018TVP}. The main technical problem is to estimate the r.h.s of~\eqref{important}. Using definiton of \(\xi_j\) we come to the problem of estimation \(\E_j |\varepsilon_{3j}|^2 |\RR_{jj}|\).  Since \(\varepsilon_{3j}\) and \(\RR_{jj}\) are dependent we need to use the Cauchy-Schwarz inequality. Unfortunately, we can only estimate \(\E_j |\varepsilon_{3j}|^2\) without truncation. To estimate higher moments we need to use truncation arguments, i.e. use \(|X_{jk}| \le n^{1/2}/\overline R\). This will lead to the non-optimal bounds. It is worth to mention that in the case when \(\beta_{4+\delta}<\infty\) we can estimate \(\E|\varepsilon_{3j}|^{2 + \delta/2}\) without truncation. To overcome the problem mentioned above we split the r.h.s. of~\eqref{important} into two terms corresponding to \(|\RR_{jj}(z)| \le H_1\) or \(|\RR_{jj}(z)| > H_1\) for some large \(H_1\).  To obtain bounds of order \(n^{-c \log n}\) for  \(\Pb(|\RR_{jj}(z)| > H_1)\) we need to take \(p\) in Lemma~\ref{main lemma} of the order \(c \log^2 n\) (\(\alpha = 2\)). 

To simplify the proof of Theorem~\ref{th: general bound} we will formulate below a simple lemma, which provides a general framework to estimate the moments of some statistics of independent random variables.

Let us consider the following statistic
\begin{eqnarray*}
T_n^{*} := \sum_{j=1}^n \xi_{j} f_{j} + \mathcal R,
\end{eqnarray*}
where  \(\xi_{j}, f_{j}, j = 1, \ldots , n\) and \(\mathcal R\) are \(\mathfrak M\)-measurable r.v. for some \(\sigma\)-algebra \(\mathfrak M\). Assume that there exist \(\sigma\)-algebras \( \mathfrak M^{(1)}, \ldots,  \mathfrak M^{(n)}, \mathfrak M^{(j)}  \subset \mathfrak M, j = 1, \ldots, n\) such that
\begin{eqnarray}\label{assumption 1}
\E_j(\xi_{j} \big |  \mathfrak M^{(j)})   = 0.
\end{eqnarray}
For simplicity we denote \(\E_j(\cdot) := \E(\cdot \big | \mathfrak M^{(j)})\). Let \(\widehat f_{j}\) be arbitrary \(\mathfrak M^{(j)}\)-measurable r.v. and denote 
\begin{eqnarray*}
\tTj := \E_j(T_n^{*}).
\end{eqnarray*}
\begin{lemma}
\label{lem: T_n general lemma}
For all \(p \geq 2\) there exist some absolute constant \(C\) such that
\begin{eqnarray*}
\E|T_n^{*}|^p \le C^p \bigg(\mathcal A + p^\frac p2 \mathcal B + p^p\mathcal C + p^p \mathcal D + \E |\mathcal R|^p \bigg), 
\end{eqnarray*}
where
\begin{eqnarray*}
\mathcal A &:=& \E\left( \frac 1n \sum_{j=1}^n \E_j|\xi_{j} (f_{j} - \widehat f_{j})|\right)^p,\\
\mathcal B &:=& \E\left(\frac 1n \sum_{j=1}^n \E_j(|\xi_j(T_n^{*} - \tTj)|) |\widehat f_{j}|\right)^\frac p 2, \\
\mathcal C &:=& \frac 1n \sum_{j=1}^n \E |\xi_{j}||T_n^{*} - \tTj|^{p-1}|\widehat f_{j}|, \\
\mathcal D &:=& \frac 1n \sum_{j=1}^n \E |\xi_{j}||f - \widehat f_{j}| |T_n^{*} - \tTj|^{p-1}.
\end{eqnarray*}
\end{lemma}
\begin{proof}
See~\cite{GotzeNauTikh2017a}[Lemma 6.1].
\end{proof}
\begin{remark}
We conclude the statement of the last lemma by several remarks. 
\begin{enumerate}
\item It follows from the definition of \( \mathcal A, \mathcal B, \mathcal C, \mathcal D\) that instead of estimation of high moments of \(\xi_j\) one needs to estimate conditional expectation \(\E_j|\xi_j|^\alpha\) for some small \(\alpha\). Typically, \(\alpha \le 4\). 
\item Moreover, to get the desired bounds one needs to choose an appropriate approximation \(\widehat f_j\) of \(f_j\) and estimate \(T_n^{*} - \tTj\). 
\end{enumerate}
\end{remark}
\begin{proof}[Proof of Theorem~\ref{th: general bound}]
Recalling the definition of \(T_n\) (see~\eqref{definition of T}) we may rewrite it in the following way
\begin{eqnarray*}
T_n = \frac{1}{n} \sum_{j=1}^n \varepsilon_{4j} \RR_{jj} + \frac{1}{n} \sum_{\nu = 1}^3 \sum_{j=1}^n \varepsilon_{\nu j} \RR_{jj}.
\end{eqnarray*}
One may see that \(T_n\) is a special case of \(T_n^{*}\), where
\begin{eqnarray*}
\xi_{j}:= \ve_{1 j} + \ve_{2j} + \ve_{3j}, \, f_{j}:= \RR_{jj}, \, \mathcal R: = \frac{1}{n} \sum_{j=1}^n \varepsilon_{4j} \RR_{jj}. 
\end{eqnarray*}	
We estimate each term in Lemma~\ref{lem: T_n general lemma}.  Here, \(\mathfrak M := \sigma\{X_{kl}, k, l \in \T\}\) and \(\mathfrak M^{(j)}: = \sigma\{ X_{kl}, k,l \in \T_j\}, j = 1, \ldots, n\). 
\subsection{Bound for \texorpdfstring{\(\E|\mathcal R|^p \)}{First term}}
Applying the Schur complement formula we get
\begin{eqnarray}\label{eq: distance between traces}
\Tr \RR - \Tr \RR^{(j)} = \big( 1 + \frac{1}{n} \sum_{k,l \in \T_j} X_{jl} X_{jk} [(\RR^{(j)})^2]_{kl} \big )\RR_{jj} = \RR_{jj}^{-1} \frac{ d \RR_{jj}}{dz}.
\end{eqnarray}
Since \( d \RR_{jj} / dz = \RR_{jj}^2\) we rewrite 
\begin{eqnarray*}\label{eq: derivative mn}
\sum_{j=1}^n \varepsilon_{4j} \RR_{jj} = \frac{1}{n}\Tr \RR^2.
\end{eqnarray*}
Hence, using Lemma~\ref{appendix lemma resolvent inequalities 1} we obtain
\begin{eqnarray*} 
|\mathcal R|^p \le \frac{1}{(nv)^p}  \E \imag^p m_n(z) \le \frac{C^p \imag^p m_{sc}(z)}{(nv)^p} + \frac{C^p \E \imag^p \Lambda_n}{(nv)^p}.
\end{eqnarray*}
We may apply the bound \( \imag \Lambda_n \le |T_n|^{1/2}\) (see~\eqref{eq: abs imag lambda main result section}), Young's inequaltity and get
\begin{eqnarray}\label{ R est main} 
|\mathcal R|^p \le \frac{1}{(nv)^p}  \E \imag^p m_n(z) \le \frac{C^p \imag^p m_{sc}(z)}{(nv)^p} + \frac{C^p}{(nv)^{2p}} + \rho \E|T_n|^p,
\end{eqnarray}
where \(0< \rho < 1\).  

\subsection{Bound for \(\mathcal A\)}
The term \( \mathcal A\), may be bounded from above by the following quantity
\begin{eqnarray*}
\max_{j=1, \ldots, n} \E \E_j^p|\xi_{j}|(\RR - \widehat f_j)|.
\end{eqnarray*}
Let us fix an arbitrary \( j = 1, \ldots, n\), and choose
\begin{eqnarray*}
\widehat f_j: = - \frac{1}{z + m_n^{(j)}(z)}. 
\end{eqnarray*}
Then
\begin{eqnarray}\label{approx of R}
\RR_{jj} - \widehat f_j = - \xi_j\, \widehat f_j\, \RR_{jj},
\end{eqnarray}
This equation implies that 
\begin{eqnarray}\label{important}
\E \E_j^p|\xi_{j}|(\RR - \widehat f_j)| = \E |\widehat f_j|^p \E_j^p |\xi_{j}|^2 |\RR_{jj}|.
\end{eqnarray}
Let us take some positive constant \(H_0>C_0\) such that for \(q = c \log^2n\):
\begin{eqnarray}\label{eq: RR condition prob}
\Pb (|\RR_{jj}|\geq H_0) \le \frac{\E|\RR_{jj}|^q}{H_0^q} \le \bigg( \frac{C_0}{H_0} \bigg)^q \le \frac{1}{n^{c_1 \log n}}.
\end{eqnarray}
It is straightforward to check that
\begin{eqnarray}\label{ A est 1}
\E |\widehat f_j|^p \E_j^p |\xi_{j}|^2 |\RR_{jj}| \one[|\RR_{jj}| \le H_0] &\le& H_0^p \E |\widehat f_j|^p \E_j^p |\xi_{j}|^2 \nonumber\\
&\le& \frac{C^p \imag^p m_{sc}(z)}{(nv)^p} + \frac{C^p}{(nv)^{2p}} + \rho \E|T_n|^p. 
\end{eqnarray}
Moreover, from~\eqref{eq: RR condition prob} and negligibility of high moments of \(\xi_j\)  we may conclude that
\begin{eqnarray}\label{ A est 2}
\E |\widehat f_j|^p \E_j^p |\xi_{j}|^2 |\RR_{jj}| \one[|\RR_{jj}| > H_0] &\le& C_0^p \Pb^{1/2}(|\RR_{jj}| > H_0) \E^\frac14 \E_j^{4p} |\xi_{j}|^2 |\RR_{jj}|\nonumber \\
&\le& C^p n^{-c_1 \log n/2} \E^\frac14 \E_j^{4p} |\xi_{j}|^2 |\RR_{jj}| \nonumber\\
&\le& C^p n^{-p}.
\end{eqnarray}
The last two inequalities~\eqref{ A est 1} and~\eqref{ A est 2} imply that
\begin{eqnarray}\label{ A est main} 
\mathcal A \le \frac{C^p \imag^p m_{sc}(z)}{(nv)^p} + \frac{C^p}{(nv)^{2p}} + \rho \E|T_n|^p. 
\end{eqnarray}

\subsection{Bound for \texorpdfstring{\(\mathcal B\)}{Second term}}
We note that
\begin{eqnarray}\label{B est}
\mathcal B^p \le \max_{j = 1, \ldots, n} \E |\widehat f_{j}|^{p/2} \E_j^{p/2} |\xi_j(T_n - \tTj)|. 
\end{eqnarray}
We estimate the conditional expectation. Let 
\begin{eqnarray}
I := \E_j (|\xi_j(T_n - \tTj)| \zeta), 
\end{eqnarray}
where \(\zeta\) is positive r.v. with sufficiently many bounded moments. For example, to estimate the r.h.s. of~\eqref{B est} one may take \(\zeta := 1\). But for further analysis it will be necessary to consider more general \(\zeta\). 

\vspace{0.1in}
\noindent
{\it Representation of \(T_n - \widetilde T_n^{(j)}\)}.  By definition we may write the following representation
\begin{eqnarray*}
T_n - \widetilde T_n^{(j)} = (\Lambda_n - \widetilde \Lambda_n^{(j)})b_n + \widetilde \Lambda_n^{(j)}(b_n - \widetilde b_n^{(j)}) - \E_j(\Lambda_n(b_n - \widetilde b_n^{(j)})),
\end{eqnarray*}
where \(\widetilde \Lambda_n^{(j)}: = \E_j (\Lambda_n)\) and \(\widetilde b_n^{(j)}:= \E_j(b_n)\). Hence,
\begin{eqnarray}\label{T - tT}
|T_n - \widetilde T_n^{(j)}| \le K^{(j)}|\Lambda_n - \widetilde \Lambda_n^{(j)}| + \frac{2}{nv} |\tTj|^{1/2}, 
\end{eqnarray}
where
\begin{eqnarray}\label{K def}
K^{(j)}: = K^{(j)}(z): = |b(z)| + 2|\tTj|^{1/2} + \frac{2}{nv}.
\end{eqnarray}
The equation~\eqref{eq: distance between traces} and Lemma~\ref{appendix lemma resolvent inequalities 1}[Inequality~\eqref{appendix lemma resolvent inequality 2}] yield that
\begin{eqnarray*}\label{eq: difference between Lambda and Lambda j}
|\Lambda_n - \Lambda_n^{(j)}| \le \frac{1}{nv} \frac{\imag \RR_{jj}}{|\RR_{jj}|}.
\end{eqnarray*}
For simplicity we denote the quadratic form in~\eqref{eq: distance between traces} by
\begin{eqnarray*}
\eta_j: = \frac{1}{n} \sum_{k,l \in \T_j} X_{jl} X_{jk} [(\RR^{(j)})^2]_{kl}
\end{eqnarray*}
and rewrite it as a sum of the three terms \(\eta_j = \eta_{0j} + \eta_{1j} + \eta_{2j}\), where
\begin{eqnarray*}
\eta_{0j}&:=& \frac{1}{n} \sum_{k \in T_j} [(\RR^{(j)})^2]_{kk} = (m_n^{(j)}(z))', \quad \eta_{1j}: = \frac{1}{n} \sum_{k \neq l \in \T_j} X_{jl} X_{jk} [(\RR^{(j)})^2]_{kl}, \\
\eta_{2j}&: =& \frac{1}{n} \sum_{k \in \T_j} [X_{jk}^2 - 1] [(\RR^{(j)})^2]_{kk}.
\end{eqnarray*}
It follows from~\eqref{eq: distance between traces} and \(\Lambda_n - \widetilde \Lambda_n^{(j)} = \Lambda_n - \Lambda_n^{(j)} - \E_j(\Lambda_n - \Lambda_n^{(j)})\) that
\begin{eqnarray*}
\Lambda_n - \widetilde \Lambda_n^{(j)} = \frac{1 + \eta_{j0}}{n} [\RR_{jj} - \E_j(\RR_{jj})] + \frac{\eta_{1j} +\eta_{2j}}{n}\RR_{jj}
- \frac{1}{n}\E_j((\eta_{j1} + \eta_{j2})\RR_{jj}).
\end{eqnarray*}
Using representation ~\eqref{approx of R} we estimate 
\begin{eqnarray*}
|\RR_{jj} - \E_j(\RR_{jj})| \le |\widehat f_j|( |\xi_j \RR_{jj}| + \E_j(|\xi_j \RR_{jj}|).
\end{eqnarray*}
Applying this inequality and Lemma~\ref{appendix lemma resolvent inequalities 1}[Inequality~\eqref{appendix lemma resolvent inequality 2}] we may write
\begin{eqnarray*}
|\Lambda_n - \widetilde \Lambda_n^{(j)}| \le \frac{1 + v^{-1}\imag m_n^{(j)}(z)}{n} |\widehat f_j|( |\xi_j \RR_{jj}| + \E_j(|\xi_j \RR_{jj}|) + \frac{|\widehat\eta_{j}\RR_{jj}|}{n}
+ \frac{\E_j(|\widehat \eta_{j}\RR_{jj}|)}{n}.
\end{eqnarray*}
Then
\begin{eqnarray*}
\E(\xi_{j} |T_n - \widetilde T_n^{(j)}||\zeta|\big|\mathfrak M^{(j)}) \le
|K^{(j)}| [B_{1j} + ... + B_{4j}] + |\tTj|^{1/2} B_{5j},
\end{eqnarray*}
where
\begin{eqnarray*}
B_{1j}&:=& \frac{1 + v^{-1}\imag m_n^{(j)}(z)}{n} |\widehat f_j| \E_j( |\xi_{ j}|^2|\RR_{jj} \zeta|), \\
B_{2j}&:=& \frac{1 + v^{-1}\imag m_n^{(j)}(z)}{n} |\widehat f_j| \E_j(|\xi_j \RR_{jj}|) \E_j( |\xi_{j} \zeta|),\\
B_{3j}&:=& \frac{1}{n}\E_j(|\xi_{ j} \widehat \eta_{j} \RR_{jj} \zeta|),\\
B_{4j}&:=& \frac{1}{n}\E_j(|\widehat \eta_{j} \RR_{jj}|) \E_j(|\xi_{j} \zeta|),\\
B_{5j}&:=& \frac{1}{n v} \E_j(|\xi_{j}\zeta| ).
\end{eqnarray*}
Hence, taking \(\zeta = 1\), we may estimate
\begin{eqnarray*}
\mathcal B  \le C^p \max_j \bigg[ \sum_{\nu=1}^{4}\E |\widehat f_{j}|^{p/2} |K^{(j)}|^{p/2} B_{\nu j}^{p/2} +  \E |\widehat f_{j}|^{p/2}  |\tTj|^{p/2}  B_{5j}^{p/2}\bigg] =: \mathcal I_1 + \ldots + \mathcal I_5. 
\end{eqnarray*}
It is straightforward to check that 
\begin{eqnarray*}
\mathcal I_5 \le \frac{\imag^{p/4} m_{sc}(z)}{(nv)^{3p/4}} \E^\frac12 |T_n|^{p} + \frac{C^p}{(nv)^{3p/4}} \E^{5p/8} |T_n|^p. 
\end{eqnarray*}
We proceed to estimation of \(\mathcal I_3\). The arguments for all other terms are similar and will be omitted. It follows from~\eqref{K def} that
Hence,
\begin{eqnarray*}
\E |\widehat f_{j}|^{p/2} |K^{(j)}|^{p/2} B_{3 j}^{p/2} &\le& C^p \E^\frac12|K^{(j)}|^p \E^\frac14 B_{3j}^{2p} \\
&\le & C^p |b(z)|^{p/2}\E^\frac14 B_{3j}^{2p} + \frac{C^p}{(nv)^{p/2}}\E^\frac14 B_{3j}^{2p} + C^p \E^{1/4} |T_n|^{p}\E^{1/4} B_{3j}^{2p}.
\end{eqnarray*}
It remains to estimate \( \E B_{3j}^{2p} \). Applying the Cauchy-Schwarz inequality and arguments similar to~\eqref{eq: RR condition prob}--\eqref{ A est 2} one may write
\begin{eqnarray*}
\E|B_{3j}|^{2p} \le \frac{1}{n^{2p}} \E^{1/2} \big[\E_j^{2p} |\widehat \eta_j|^2 \big]\E^{1/2}\big[ \E_j^{2p}|\xi_j|^2 |\RR_{jj}|^2\big] \le \frac{C^p \imag^{2p} m_{sc}(z)}{(nv)^{4p}} + \frac{C^p}{(nv)^{4p}} \E|T_n|^p. 
\end{eqnarray*} 
Here we also use the moment bounds for quadratic and linear forms, see~\cite{GotzeNauTikh2015a}[Lemmas A.3--A.12]. Repeating the same arguments for \(\mathcal I_1, \mathcal I_2, \mathcal I_3\) we  come to the following bound
\begin{eqnarray*}
\mathcal B  &\le& \frac{C^p |b(z)|^{p}}{(nv)^p}  + \frac{C^p|b(z)|^{p/2}}{(nv)^p}\E^\frac14 |T_n|^p + \frac{C^p \imag^{p/2} m_{sc}(z)}{(nv)^{3p/2}} + \frac{C^p}{(nv)^{3p/2}} \E^\frac14 |T_n|^p \\
&+& \frac{C^p \imag^{p/2} m_{sc}(z)}{(nv)^p} \E^\frac14 |T_n|^p + \frac{C^p}{(nv)^p} \E^\frac12 |T_n|^p.
\end{eqnarray*}
where we also used the crude estimate \(\imag m_{sc}(z) \le |b(z)| \). Using Young's inequality we immediately obtain
\begin{eqnarray}\label{ B est main} 
p^{p/2} \mathcal B \le \frac{C^p p^{p/2} |b(z)|^{p}}{(nv)^p}  + \frac{C^p p^p}{(nv)^{2p}} +  \rho \E|T_n|^p.
\end{eqnarray}
\subsection{Bound for \texorpdfstring{\(\mathcal C\) and \(\mathcal D\)}{Third term}} 
It is easy to see that one may estimate \(\mathcal C\) and \(\mathcal D\) simultaneously. Indeed, it is enough to estimate
\begin{eqnarray*}
C' : = \max_{j=1, \ldots, n} \E |\xi_{j}||T_n - \tTj|^{p-1}|\zeta|,
\end{eqnarray*}
where \(\zeta: = \max(|f_j|,|\widehat f_j|)\). Let us fix \(j = 1, \ldots,, n\). Using~\eqref{T - tT} we get
\begin{eqnarray*}
\E |\xi_{j}||T_n - \tTj|^{p-1}|\zeta| &\le& \frac{C^p |b(z)|^{p-2}}{(nv)^{p-2}}\E \E_j|\xi_{j}| |T_n - \tTj| |\zeta| \\
&+& \frac{C^p}{(nv)^{p-2}} \E |\tTj|^\frac{p-2}{2} \E_j|\xi_{j}| |T_n - \tTj| |\zeta|.
\end{eqnarray*} 
Applying Young's inequality we obtain
\begin{eqnarray*}
p^{p} \max(\mathcal C, \mathcal D) \le \frac{C^p p^p |b(z)|^p}{(nv)^{p}} + \frac{C^p p^{2p} }{(nv)^{2p}}  + p^{p/2} \max_{j = 1, \ldots, n} \E \E_j^{p/2} |\xi_j(T_n - \tTj)||\zeta| + \rho \E |T_n|^p.
\end{eqnarray*}
Similarly to~\eqref{eq: RR condition prob}
\begin{eqnarray*}
\Pb(|\zeta| \geq H_0) \le n^{-c \log n}.
\end{eqnarray*}
Repeating now all calculations above we get
\begin{eqnarray}\label{ CD est main} 
\max( \mathcal C, \mathcal D ) \le \frac{C^p p^p |b(z)|^{p}}{(nv)^p} +  \frac{C^p p^{2p}}{(nv)^{2p}} + \rho \E |T_n|^p.
\end{eqnarray}
Collecting~\eqref{ R est main}, \eqref{ A est main}, \eqref{ B est main}  and \eqref{ CD est main} we conclude the claim of the Theorem~\ref{th: general bound}. 
\end{proof}

\section{Acknowledgements} 
We would like to thank the Associate Editor and the Reviewer for helpful comments and suggestions.

Results have been obtained under support of the RSF grant No. 18-11-00132 (HSE University).  F. G\"otze has been supported by DFG through the Collaborative Research Centres 1283  "Taming uncertainty and profiting from randomness and low regularity in analysis, stochastics and their applications".

\appendix
\section{Auxiliary results}
\subsection{Truncation}

In this section we will show that the  conditions \(\Cond\) allows us to assume that  for all \(1 \le j,k \le n\) we have \(|X_{jk}| \le \sqrt n/\overline R\), where \(\overline R\) is some positive constant .

Let \(\hat X_{jk}: = X_{jk} \one[|X_{jk}| \leq \sqrt{n}/\overline R]\), \(\tilde X_{jk}: = X_{jk} \one[|X_{jk}| \geq  \sqrt{n}/\overline R] - \E X_{jk} \one[|X_{jk}| \geq \sqrt n/\overline R ]\) and finally
\(\breve X_{jk}: = \tilde X_{jk} \sigma^{-1}\), where \(\sigma^2: = \E |\tilde X_{11}|^2\). We denote symmetric random matrices by \(\hat \X, \tilde \X\) and \(\breve \X\) formed from \(\hat X_{jk}, \tilde X_{jk}\) and \(\breve X_{jk}\) respectively. Similar notations are used for the corresponding resolvent matrices, ESD and Stieltjes transforms.
\begin{lemma}\label{appendix: lemma trunc 1} Assuming the conditions \(\Cond\) we have for all \(1 \le p \le A_1 \log n\)
\begin{eqnarray*}
\E|m_n(z) - \hat m_n(z)|^p \le \left(\frac{C \overline R^4}{nv}\right)^p.
\end{eqnarray*}
\end{lemma}
\begin{proof}
From Bai's rank inequality (see~\cite{BaiSilv2010}[Theorem~A.43]) we conclude that
\begin{eqnarray*}
	\sup_{x \in \R} |\mu_n((-\infty, x]) - \hat {\mu}_n((-\infty, x])| \le \frac{1}{n} \Rank(\X - \hat \X) \le \frac{1}{n} \sum_{j,k=1}^n \one[|X_{jk}| \geq \sqrt n/\overline R].
\end{eqnarray*}
Integrating by parts we get
\begin{eqnarray*}
	\E|m_n(z) - \hat m_n(z)|^p \le \frac{1}{(nv)^p} \E \left(\sum_{j,k=1}^n \one[|X_{jk}| \geq  \sqrt n/\overline R] \right)^p.
\end{eqnarray*}
It is easy to see that
\begin{eqnarray*}
\left(\sum_{j,k=1}^n \E\one[|X_{jk}| \geq \sqrt n /\overline R] \right)^p \le C^p \overline R^{4p}.
\end{eqnarray*}
Applying Rosenthal's inequality, \cite{Rosenthal1970}, we get that
\begin{eqnarray*}
&&\E \left(\sum_{j,k=1}^n [\one[|X_{jk}| \geq \sqrt n/\overline R] - \E\one[|X_{jk}| \geq  \sqrt n/\overline R]] \right)^p \\
&&\qquad\qquad\qquad\le C^p \left( \left(\frac{p \overline R^4}{n^2} \sum_{j,k=1}^n \E |X_{jk}|^{4} \right)^\frac{p}{2}  + \frac{p^p \overline R^4}{n^2} \sum_{j,k=1}^n \E |X_{jk}|^{4}  \right ) \le (C \sqrt p \overline R^2)^p.
\end{eqnarray*}
From these inequalities we may conclude the statement of Lemma.
\end{proof}

\begin{lemma}\label{appendix: lemma trunc 0}
Assuming the conditions \(\Cond\) we have for all \(1 \le p \le A_1 \log n\)
\begin{eqnarray*}
\E|\tilde m_n(z) - \breve m_n(z)|^p \le \frac{(C \overline R^2)^p \mathcal A^p(2p)}{(nv)^p}.
\end{eqnarray*}
\end{lemma}
\begin{proof}
It is easy to see that
\begin{eqnarray}\label{eq: tilde R representation}
\tilde \RR(z) = (\tilde \W - z\I)^{-1} = \sigma^{-1} (\breve \W - z \sigma^{-1}\I)^{-1}  = \sigma^{-1}\breve \RR(\sigma^{-1}z).
\end{eqnarray}
Applying the resolvent equality we get
\begin{eqnarray}\label{eq: overline R representation}
\breve \RR(z) - \breve \RR(\sigma^{-1}z) = (z - \sigma^{-1}z) \breve \RR(z) \breve \RR(\sigma^{-1}z).
\end{eqnarray}
From~\eqref{eq: tilde R representation} and~\eqref{eq: overline R representation} we may conclude
\begin{eqnarray*}
|\tilde m_n(z) - \breve m_n(z)| &=& \frac{1}{n} | \Tr \tilde \RR(z) - \Tr \breve \RR(z)| = \frac{1}{n} | \sigma^{-1}\Tr \breve \RR(\sigma^{-1}z) - \Tr \breve \RR(z)| \\
&=& \frac{1}{n} | \sigma^{-1}\Tr \breve \RR(z) - \Tr \breve \RR(z)- (z - \sigma^{-1}z) \Tr \breve \RR(z) \breve \RR(\sigma^{-1}z)|\\
&\le& \frac{1}{n}(\sigma^{-1} - 1) |\Tr \breve \RR(z)| + (\sigma^{-1} - 1)  \frac{|z|}{n}|\Tr \breve \RR(z) \breve \RR(\sigma^{-1}z)|.
\end{eqnarray*}
Taking the \(p\)-th power and mathematical expectation we get
\begin{eqnarray*}
\E|\tilde m_n(z) - \breve m_n(z)|^p \le  \frac{1}{n^p}(\sigma^{-1} - 1)^p \E|\Tr \breve \RR(z)|^p + (\sigma^{-1} - 1)^p  \frac{C^p}{n^p}\E|\Tr \breve \RR(z) \breve \RR(\sigma^{-1}z)|^p.
\end{eqnarray*}
Since \(\breve \X\) satisfies conditions \(\CondTwo\) we may apply  Lemma~\ref{main lemma} and conclude
\begin{eqnarray*}
\frac{1}{n^p} \E|\Tr \breve \RR(z)|^p \le C_0^p.
\end{eqnarray*}
We also have
\begin{eqnarray}\label{appendix variance truncation}
\sigma^{-1} - 1 \le \sigma^{-1} (1  - \sigma) \le \sigma^{-1} (1  - \sigma^2) \le \sigma^{-1} \E |X_{jk}|^2 \one[|X_{jk}| \geq \sqrt n/\overline R]   \le C \overline R^2/n.	
\end{eqnarray}
To finish the proof it remains to estimate the term
\begin{eqnarray*}
\frac{1}{n^p}\E|\Tr \breve \RR(z) \breve \RR(\sigma^{-1}z)|^p.
\end{eqnarray*}
Applying the obvious inequality \(|\Tr \A \B| \le \|\A\|_2 \| \B\|_2\) we get
\begin{eqnarray*}
\frac{1}{n^p}\E|\Tr \breve \RR(z) \breve \RR(\sigma^{-1}z)|^p &\le& \frac{1}{n^p} \E^\frac12\|\breve \RR(z)\|_2^{2p} \E^\frac{1}{2}\| \breve \RR(\sigma^{-1}z) \|_2^{2p}\\
&\le& \frac{\E^\frac12 \imag^p \breve m_n(z) \E^\frac12 \imag^p \breve m_n(\sigma^{-1} z)}{v^p}.
\end{eqnarray*}
From this inequality and~\eqref{appendix variance truncation} we conclude the statement of the lemma.
\end{proof}

\begin{lemma}\label{appendix: lemma trunc 2}
Assuming the conditions \(\Cond\) we have for all \(1 \le p \le A_1 \log n\):
\begin{eqnarray*}
\E|\tilde m_n(z) - \hat m_n(z)|^p \le \frac{(C\overline R^3)^p}{(nv)^{3p/2}}.
\end{eqnarray*}
\end{lemma}
\begin{proof}
It is easy to see that
\begin{eqnarray*}
\tilde m_n(z) - \hat m_n(z)  = \frac{1}{n} \Tr (\tilde \W - \hat \W) \hat \RR \tilde \RR.
\end{eqnarray*}
Applying the obvious inequalities \(|\Tr \A \B| \le \|\A\|_2 \| \B\|_2\) and \(\|\A \B\|_2 \le \|\A\|\|\B\|_2\) we get
\begin{eqnarray*}
|\tilde m_n(z) - \hat m_n(z)|  \le \|\tilde \W - \hat \W\|_2 \|\hat \RR\|_2 \|\tilde \RR\| = \|\E \hat \W\|_2 \|\tilde \RR\|_2 \|\hat \RR\|.
\end{eqnarray*}
From
\begin{eqnarray*}
|\E \hat X_{jk}| = |\E X_{jk} \one[|X_{jk}| \geq \sqrt n/\overline R]| \le \frac{C \overline R^3}{n^{3/2}}
\end{eqnarray*}
we obtain
\begin{eqnarray*}
\|\E \hat \W\|_2 \le \frac{C \overline R^3}{n}.
\end{eqnarray*}
By Lemma~\ref{appendix: lemma trunc 0} we know \(\E |\tilde m_n(z)|^p \le C^p\). This implies that
\begin{eqnarray*}
\frac{1}{n^\frac{p}{2}} \E\|\tilde \RR\|_2^p \le \frac{C^p}{v^\frac{p}{2}}.
\end{eqnarray*}
Finally
\begin{eqnarray*}
\E|\tilde m_n(z) - \hat m_n(z)|^p \le\frac{C^p \overline R^{3p} }{(nv)^\frac{3p}{2}}.
\end{eqnarray*}
\end{proof}

\subsection{Replacement}

We say that the conditions $\CondThree$ are satisfied if $X_{jk}$ satisfies the conditions $\Cond$ and have a sub-Gaussian distribution. It is well-known that the random variables $\xi$ are sub-gaussian if and only if $\E^{1/p} |\xi|^p  \le C \sqrt{p}$ for some constant \(C>0\). 
\begin{lemma}\label{lem: bound in the optimal region}
For all \(v \geq v_0\)  and \(5 \le p \le \log n\) there exist positive constants \(C_1, C_2\) such that
\begin{eqnarray*}
\E|\G_{jk}(v)|^p \le C_1^p + C_2\E|\G_{jk}^\y(v)|^p,
\end{eqnarray*}
where \(G_{jk}^\y\) is defined in~\eqref{sub gaussian matrix}. 
\end{lemma}
\begin{proof}
The method is based on the following replacement scheme, which has been used in  recent results~\cite{ErdKnowYauYin2012},~\cite{LeeYin2014} and~\cite{GotzeNauTikh2016a}. We replace  all \(h_{ab}\) by \(\overline h_{ab}\) for \((a,b)\) such that \(L_{ab} = 1\), thus  replacing the corresponding  resolvent entries \(\G_{jk}\) by \( \G_{jk}^\y\) for every pair of \((j, k)\). Let \(\J, \K \subset \T\). Denote by \(\HH^{(\J, \K)}\) the  random matrix \(\HH\) with all entries in the positions \((\mu, \nu), \mu \in \J, \nu \in \K\) replaced by \(\overline \xi_{\mu \nu}\). Assume that we have already exchanged all entries in positions \((\mu, \nu), \mu \in \J, \nu \in \K\) and are going to replace an additional entry in the position \((a, b), a \in \T \setminus \J, b \in \T \setminus \K\) with \( L_{ab} = 1\). Without loss of generality we may assume that \(\J = \emptyset, \K = \emptyset\) (hence \(\HH^{(\J, \K)} = \HH\)) and then denote \(\V: = \HH^{(\{a\}, \{b\})}\). The following  additional notations will be needed.
\begin{eqnarray*}
\EE^{(a,b)} =
\begin{cases}
\ee_{a} \ee_b^\mathsf{T} + \ee_b \ee_a^\mathsf{T}, & 1 \le a < b \le n, \\
\ee_{a} \ee_a^\mathsf{T}, &a = b.
\end{cases}
\end{eqnarray*}
and \(\U : = \HH - \EE^{(a,b)}\), where \(\ee_j\) denotes a unit column-vector with all zeros except \(j\)-th position. In these notations we may write
\begin{eqnarray*}
\HH = \U +  \xi_{ab} \EE^{(a,b)}, \quad \V = \U + \overline \xi_{ab} \EE^{(a,b)}.
\end{eqnarray*}
Recall that \(\G: = (n^{-1/2}\HH - z \I)^{-1}\) and denote \(\RS: = (\V - z \I)^{-1}\) and \(\RT: = (\U - z \I)^{-1}\). Let us assume that we have already proved the following fact
\begin{eqnarray}\label{eq: main eq 1}
\E|\G_{jk}|^p = \mathcal I(p) + \frac{\theta_1 C^p}{n^2} + \frac{\theta_1\E|\G_{jk}|^p}{n^2},
\end{eqnarray}
where \(\mathcal I(p)\) is some quantity depending on \(p, n\) (see~\eqref{I term} below for precise definition)  and \(|\theta_1|\le 1, C > 0\) are some numbers. Similarly,
\begin{eqnarray}\label{eq: main eq 2}
\E|\RS_{jk}|^p = \mathcal I(p) + \frac{\theta_2 C^p}{n^2} + \frac{\theta_2\E|\RS_{jk}|^p}{n^2},
\end{eqnarray}
where \(|\theta_2| \le 1\). It follows from~\eqref{eq: main eq 1} and~\eqref{eq: main eq 2} that
\begin{eqnarray*}
\left(1 - \frac{\theta_1}{n^2}\right )\E|\G_{jk}|^p \le \left(1 - \frac{\theta_2}{n^2}\right )\E|\RS_{jk}|^p + \frac{2C^p}{n^2}.
\end{eqnarray*}
Let us denote \(\rho : = \left(1 - \theta_2/n^2 \right )\left(1 -\theta_1/n^2\right )^{-1}\). We get
\begin{eqnarray}\label{eq: step}
\E|\G_{jk}|^p \le \rho \E|\RS_{jk}|^p + C_1^p/n^2,
\end{eqnarray}
with some positive constant \(C_1\). Repeating~\eqref{eq: step} recursively for \((a,b): L_{ab} = 1\) we arrive at the following bound
\begin{eqnarray}\label{eq: step 2}
\E|\G_{jk}|^p \le \rho^\frac{n(n+1)}{2} \E|\G_{jk}^\y|^p + \frac{C_1^p}{n^2} \left(1 + \rho_1 + ... + \rho_1^{M - 1 }\right),
\end{eqnarray}
where \( M \le n(n+1)/2\). It is easy to see from the definition of \(\rho\) that for some \(\theta\), say \(|\theta| < 4\), we have
\begin{eqnarray*}
\rho \le 1 + |\theta|/n^2.
\end{eqnarray*}
From this inequality and~\eqref{eq: step 2} we deduce that
\begin{eqnarray*}
\E|\G_{jk}|^p \le C_2 \E|\G_{jk}^\y|^p + C_3^p,
\end{eqnarray*}
with some positive constants \(C_2\) and \(C_3\). From the last inequality we may conclude the statement of the lemma. It remains to prove~\eqref{eq: main eq 1} (resp.~\eqref{eq: main eq 2}).  Applying the resolvent equation we get for \(m \geq 0\)
\begin{eqnarray}\label{resolvent expansion R}
\G = \RT + \sum_{\mu=1}^m \frac{(-1)^\mu}{n^\frac{\mu}{2}} \xi_{ab}^\mu (\RT \EE^{(a,b)})^\mu \RT + \frac{(-1)^{m+1}}{n^\frac{m+1}2} \xi_{ab}^{m+1} (\RT \EE^{(a,b)})^{m+1} \G.
\end{eqnarray}
The same identity holds for \(\RS\)
\begin{eqnarray*}\label{resolvent expansion S}
\RS = \RT + \sum_{\mu=1}^m \frac{(-1)^\mu}{n^\frac{\mu}{2}} \overline \xi_{ab}^\mu (\RT \EE^{(a,b)})^\mu \RT + \frac{(-1)^{m+1}}{n^\frac{m+1}{2}} \overline \xi_{ab}^{m+1} (\RT \EE^{(a,b)})^{m+1} \RS.
\end{eqnarray*}
We investigate~\eqref{resolvent expansion R}. In order handle arbitrary high moments of \(\G_{jk}\) we apply a Stein type technique similar to Theorem. Let us introduce the following function \(\varphi(z): = \overline z |z|^{p-2}\) and write
\begin{eqnarray*}
\E |\G_{jk}|^p = \E \G_{jk} \varphi(\G_{jk}).
\end{eqnarray*}
Applying~\eqref{resolvent expansion R}  we get
\begin{eqnarray*}
\E |\RR_{jk}|^p &=& \sum_{\mu=0}^4 \frac{(-1)^\mu}{n^\frac{\mu}{2}} \E \xi_{ab}^\mu [(\RT \EE^{(a,b)})^\mu \RT]_{jk} \varphi(\G_{jk})
\nonumber\\
&+&\sum_{\mu=5}^m \frac{(-1)^\mu}{n^\frac{\mu}{2}} \E \xi_{ab}^\mu [(\RT \EE^{(a,b)})^\mu \RT]_{jk} \varphi(\G_{jk}) \nonumber \\
&+&  \frac{1}{n^\frac{m+1}2} \E \xi_{ab}^{m+1} [(\RT \EE^{(a,b)})^{m+1} \G]_{jk}\varphi(\G_{jk}) =: \mathcal A_0 + \mathcal A_1 + \mathcal A_2.
\end{eqnarray*}
Repeating the arguments from~\cite{GotzeNauTikh2016a} one may show that 
\begin{eqnarray*}
\max (|\mathcal A_1|, |\mathcal A_2|) \le \frac{C^p}{n^2} + \rho \frac{\E |\G_{jk}|^p}{n^2}. 
\end{eqnarray*}
For the term \( \mathcal A_{00}\) one may write down the following representation
\begin{eqnarray*}
\mathcal A_0 = \mathcal I(p) + r_n(p),
\end{eqnarray*}
with the remainder term bounded in absolute value
\begin{eqnarray*}
|r_n(p)| \le \frac{C^p}{n^2} + \rho \frac{\E |\G_{jk}|^p}{n^2},
\end{eqnarray*}
and
\begin{eqnarray}\label{I term}
\mathcal I(p) := \sum_{\mu=0}^4 \frac{(-1)^\mu}{n^\frac{\mu}{2}} \E \xi_{ab}^\mu \E [(\RT \EE^{(a,b)})^\mu \RT]_{jk} \varphi(\RT_{jk}) + \sum_{\mu=0}^4 \sum_{l=1}^{4-\mu} \frac{(-1)^\mu}{l!} \mathcal B_{\mu l}^{(0)},
\end{eqnarray}
where 
\begin{eqnarray*}
\mathcal B_{\mu l}^{(0)} &:=& \sum_{\substack{\mu_1 + ... + \mu_m = l \\ \mu+\mu_1 + 2\mu_2 + ... + m \mu_m \le 4 }} \frac{C_{\mu_1, ... , \mu_m}^l}{n^{\frac{\mu}{2} + \frac{\mu_1}{2} + \frac{2\mu_2}{2} + ... + \frac{m \mu_m}{2}}} \E X_{ab}^{\mu+\mu_1 + 2\mu_2 + ... + m \mu_m} \\
&&\qquad\qquad\times \E [(\RT \EE^{(a,b)})^\mu \RT]_{jk}[(\RT \EE^{(a,b)})\RT]_{jk}^{\mu_1}...[(\RT \EE^{(a,b)})^m\RT]_{jk}^{\mu_m} \varphi^{(l)}(\RT_{jk}).
\end{eqnarray*}
One may see that the term \(\mathcal I(p)\) doesn't depend on \(\G\) but depends on \( \RT \).
\end{proof}

\begin{lemma}\label{lemma step for resolvent subgauss} Let \(\LL\) be \(r\)-admissible and assume that the conditions \(\CondThree\) hold. Let \(C_0\) and \(s_0\) be arbitrary numbers such that \(C_0 \geq \max(1/V, 6 c_0), s_0 \geq 2\). There exist a sufficiently large constant \(A_0\) and small constant \(A_1\) depending on \(C_0, s_0, V\) only such that the following statement holds. Fix some \(\tilde v: \tilde v_0 s_0 \leq \tilde v \le V\). Suppose that for some integer \(L > 0\), all \(u, v',q\) such that  \(\tilde v \leq v' \leq V,\, |u| \le u_0, 1 \le q \le A_1 (n v')\)
\begin{eqnarray*}
\max_{\J: |\J| \le L} \max_{l, k \in \T_\J}\E |\G_{l k}^{(\J)}(v')|^q \le C_0^q.
\end{eqnarray*}
Then for all \(u,v, q\) such that \(\tilde v/s_0 \leq v \le V, |u| \le u_0\), \(1 \le q \le A_1 (n v)\)
\begin{eqnarray*}
\max_{\J: |\J| \le L-1} \max_{l, k \in \T_\J}\E |\G_{l k}^{(\J)}(v)|^q \le C_0^q.
\end{eqnarray*}
\end{lemma}
\begin{proof}
We first observe the fact that the factor \(q\) appears only in the terms with \(\overline \xi_{jk}\). Let us consider only one term, for example, \( \):
\begin{eqnarray*}
\zeta_{2j}:= -\frac{2}{n}\sum_{l \ne k \in \T_{\J,j} \setminus \mathbb A_j}  (\overline \xi_{jl} - \E \overline \xi_{jl}) (\overline \xi_{jk} - \E \overline \xi_{jk}) \G_{kl}^{(\J, j)}.
\end{eqnarray*} 
Applying the Hanson-Wright inequality, see e.g.~\cite{RudelsonVershynin2013} we obtain that
\begin{eqnarray*}
\E |\zeta_{2j}|^{2q} \le \frac{C^q q^q \E \imag^q m_n^{(\J,j)}}{(nv)^q} + \frac{C^q q^{2q}}{(nv)^{2q}} \le \frac{C^q q^q (C_0 s_0)^q }{(nv)^q}.
\end{eqnarray*}
\end{proof}

\subsection{Inequalities for resolvent}
\begin{lemma}\label{appendix lemma resolvent inequalities 1}
For any \(z = u + i v \in \C^{+}\) we have
\begin{eqnarray}\label{appendix lemma resolvent inequality 1}
\frac{1}{n} \sum_{l,k \in \T_{\J}} |\RR_{kl}^{(\J)}|^2 \le \frac{1}{v} \imag m_n^{(\J)}(z).
\end{eqnarray}
For any \(l \in \T_{\J}\)
\begin{eqnarray}\label{appendix lemma resolvent inequality 2}
\sum_{k \in \T_{\J}} |\RR_{kl}^{(\J)}|^2 \le \frac{1}{v} \imag \RR_{ll}^{(\J)}.
\end{eqnarray}
\end{lemma}

\def\polhk#1{\setbox0=\hbox{#1}{\ooalign{\hidewidth
			\lower1.5ex\hbox{`}\hidewidth\crcr\unhbox0}}}

\end{document}